\documentclass[10pt,twoside,reqno]{tran-l}

\newtheorem{theorem}{Theorem}[section]
\newtheorem{lemma}[theorem]{Lemma}

\newtheorem{definition}[theorem]{Definition}

\newtheorem{remark}[theorem]{Remark}
\numberwithin{equation}{section}

\usepackage{amsmath}  	
\usepackage{latexsym}	
\usepackage{amssymb}  	
\usepackage{amsaddr}      
\usepackage{amsthm}   	
\usepackage[square, numbers]{natbib}
\usepackage{graphicx}
\usepackage{float}            
\DeclareMathOperator{\sgn}{sgn}
\usepackage{mathrsfs}	
\usepackage[all]{xy}	
\usepackage{mathtools}

\usepackage[T1,hyphens]{url}
\usepackage{pdflscape}
\usepackage{multirow}
\usepackage{pifont}          
\usepackage{changebar}
\usepackage{fancyhdr}
\usepackage{ecltree}
\usepackage{placeins}
\usepackage{relsize,exscale}
\usepackage{bm}
\usepackage{chngcntr}
\usepackage[x11names]{xcolor}
\usepackage[colorlinks=true,linkcolor=blue,citecolor=SpringGreen4,urlcolor=blue]{hyperref}
\usepackage{rotating}
\usepackage{comment}
\usepackage{enumerate}
\usepackage{dsfont}   
\usepackage{tikz-cd}
\usepackage{pgfplots}

\renewcommand{\leq}{\leqslant}

\renewcommand{\geq}{\geqslant}

\usepackage[noend]{algpseudocode}
\usepackage{etoolbox}
\usepackage[utf8]{inputenc}
\usepackage[ruled, lined, longend, linesnumbered]{algorithm2e}

\usepackage{listings}

\makeatletter
\renewcommand{\tocsection}[3]{%
  \indentlabel{\@ifnotempty{#2}{\bfseries\ignorespaces#1 #2\quad}}\bfseries#3}
\renewcommand{\tocsubsection}[3]{%
  \indentlabel{\@ifnotempty{#2}{\ignorespaces#1 #2\quad}}#3}
\newcommand\@dotsep{4.5}
\def\@tocline#1#2#3#4#5#6#7{\relax
  \ifnum #1>\c@tocdepth 
  \else
    \par \addpenalty\@secpenalty\addvspace{#2}%
    \begingroup \hyphenpenalty\@M
    \@ifempty{#4}{%
      \@tempdima\csname r@tocindent\number#1\endcsname\relax
    }{%
      \@tempdima#4\relax
    }%
    \parindent\z@ \leftskip#3\relax \advance\leftskip\@tempdima\relax
    \rightskip\@pnumwidth plus1em \parfillskip-\@pnumwidth
    #5\leavevmode\hskip-\@tempdima{#6}\nobreak
    \leaders\hbox{$\m@th\mkern \@dotsep mu\hbox{.}\mkern \@dotsep mu$}\hfill
    \nobreak
    \hbox to\@pnumwidth{\@tocpagenum{\ifnum#1=1\bfseries\fi#7}}\par
    \nobreak
    \endgroup
  \fi}
\AtBeginDocument{%
\expandafter\renewcommand\csname r@tocindent0\endcsname{0pt}
}
\def\l@subsection{\@tocline{2}{0pt}{2.5pc}{5pc}{}}

\let\oldtocsubsubsection=\tocsubsubsection
\renewcommand{\tocsubsubsection}[2]{\hspace{3em}\oldtocsubsubsection{#1}{#2}}

\makeatother

\begin{document}

\title[BGC Stochastic Processes' Link with M-SBM]{Bi-Directional Grid Constrained Stochastic Processes' Link to Multi-Skew Brownian Motion}

\author{Aldo Taranto, Ron Addie, Shahjahan Khan}





\email{Aldo.Taranto@, Ron.Addie@, Shahjahan.Khan@, \quad @usq.edu.au}

\thanks{The first author was supported by an Australian Government Research Training Program (RTP) Scholarship.}
\thanks{We would like to thank the Editors of this journal for their positive feedback and invaluable advice on refining this paper.}

\address{\text{ }\\
School of Sciences\\
University of Southern Queensland\\
Toowoomba, QLD 4350, Australia}

\date{\today}



\begin{abstract}
Bi-Directional Grid Constrained (BGC) stochastic processes (BGCSPs) constrain the random movement toward the origin steadily more and more, the further they deviate from the origin, rather than all at once imposing reflective barriers, as does the well-established theory of It\^{o} diffusions with such reflective barriers.
We identify that BGCSPs are a variant rather than a special case of the multi-skew Brownian motion (M-SBM).
This is because they have their own complexities, such as the barriers being hidden (not known in advance) and not necessarily constant over time.
We provide a M-SBM theoretical framework and also a simulation framework to elaborate deeper properties of BGCSPs.
The simulation framework is then applied by generating numerous simulations of the constrained paths and the results are analysed.
BGCSPs have applications in finance and indeed many other fields requiring graduated constraining, from both above and below the initial position.
\end{abstract}

\maketitle 

\tableofcontents



\section*{Notation}

\begin{table}[H]
   \begin{center}
\begin{tabular}{  l  l  }
\hline
 \textbf{Term} & \textbf{Description}   \\ \hline
BGC                & Bi-Directional grid constrained\\
BGCSP             & BGC stochastic process\\
M-SBM             & Multi-skew Brownian motion\\
$X_t$              & Stochastic process over time $t$\\
$B_t$              & Brownian motion over time $t$\\
$W_t$              & Wiener process over time $t$ \\
$T$                 & Time\\
$f(X_t, t)$, $\mu$             & Drift coefficient of $X$ over $t$\\
$g(X_t, t)$, $\sigma$         & Diffusion coefficient of $X$ over $t$\\
$\Psi(X_t, t)$         & BGC coefficient of $X$ over $t$\\
$\sgn[X_t]$         & Sign function of $X_t$\\
$\mathfrak{B}_L$  &  Lower barrier\\
$\mathfrak{B}_U$  &  Upper barrier\\
$\tau_{\mathfrak{B} }$             & Stopping time when $\mathfrak{B}$ is hit\\
$| x |$            & Absolute value of $x$\\
\hline
\end{tabular}
   \label{Tab:Notation}
   \end{center}
\end{table}

\bigskip
\section{Introduction}

\noindent
In \citet{TarantoKhan2020_1}, the concept of Bi-Directional Grid Constrained (BGC) stochastic processes (BGCSPs) was introduced, and the impact of BGC on the iterated logarithm bounds of the corresponding stochastic differential equation (SDE) was derived.
In the subsequent paper (\citet{TarantoKhan2020_2}), the hidden geometry of the BGC process was examined, in particular how the hidden reflective barriers are formed, the various possible formulations that can be formed, and an algorithm to simulate BGCSPs was provided.
BGCSPs were also compared to the Langevin equation, the Ornstein-Uhlenbeck process (OUP), and it was shown how BGCSPs are more complex due to a more general framework that does not always admit an exact solution.
In this paper, we examine in what respects a BGCSP resembles and in what respect it differs from a type of multi-skew Brownian motion (M-SBM).

\bigskip \noindent
Before we do this, we recall that the BGC term $\sgn[X_t] \Psi(X_t, t)$, from now onwards simply $\Psi(X_t, t)$, unless required to be fully expressed, has been defined as impacting the $dt$ term as in \cite{TarantoKhan2020_1} or the $dW_t$ term to a lesser extent as in \cite{TarantoKhan2020_2}.
Here, we notice that even if this term is placed as a factor outside the sum of the $dt$ and $dW_t$ terms, it will still constrain the It\^{o} diffusion in a slightly different, yet fundamentally the same way.
Since $\Psi (x)$ is continuous and differentiable, we require $\Psi (x)$ to include the infinitessimal $d$, giving $d\Psi (x)$.
We thus provide the following third alternative definition of the BGCSP.

\bigskip \noindent
\begin{definition}\textbf{(Definition III of BGC Stochastic Processes)}. For a complete filtered probability space $(\Omega, \mathcal{F}, \{ \mathcal{F} \}_{t \geq 0}, \mathbb{P})$ and a (continuous) BGC function $\Psi (x) : \mathbb{R} \rightarrow \mathbb{R}$, $\forall x \in \mathbb{R}$, $f(X_t, t)$ is the drift coefficient, $g(X_t, t)$ is the diffusion coefficient, $\Psi(X_t, t)$ is the BGC function.
$f(X_t, t)$, $g(X_t, t)$ and $\Psi (X_t, t)$ are convex functions and $\sgn[x]$ is the sign function defined in the usual sense,

\bigskip
\begin{eqnarray*}
   \sgn [x] & = &
{
   \begin{cases}
\displaystyle  \phantom{-}1 &, \quad x > 0 \\
\displaystyle  \phantom{-}0 &, \quad x = 0 \\
\displaystyle                   -1 &, \quad x < 0
   \end{cases}}.
\end{eqnarray*}

\bigskip \noindent
Then, the corresponding BGC It\^{o} diffusion is defined as follows,

\bigskip
\begin{equation}\label{Eq:BGC}
      dX_t  =  f(X_t, t) \, dt + g(X_t, t) \, dW_t  -  \underbrace{ \sgn[X_t] \, d\Psi  (X_t, t) }_{\textbf{BGC}}.
\end{equation}
         \hfill    $\blacksquare$
\end{definition}

\bigskip \noindent
By choosing $\Psi (X_t, t)$ to be a `suitable' parabolic cylinder function in relation to the underlying It\^{o} diffusion, as shown in \cite{TarantoKhan2020_2}, namely $\Psi(X_t, t) = \left( \frac{X_t}{10} \right)^2$, the hidden barriers become visible when enough simulations hit them.
Figure \ref{Fig:SamplePositiveGrowthPathofGridTrader3of3} shows that (\ref{Eq:BGC}) is simulated 1,000 times for both with and without BGC, together with the hidden lower reflective barrier $a = \mathfrak{B}_L$ and the hidden reflective upper barrier $b = \mathfrak{B}_U$ also being displayed.

\begin{figure}[htb]
  \centering
   \includegraphics[width=\linewidth]{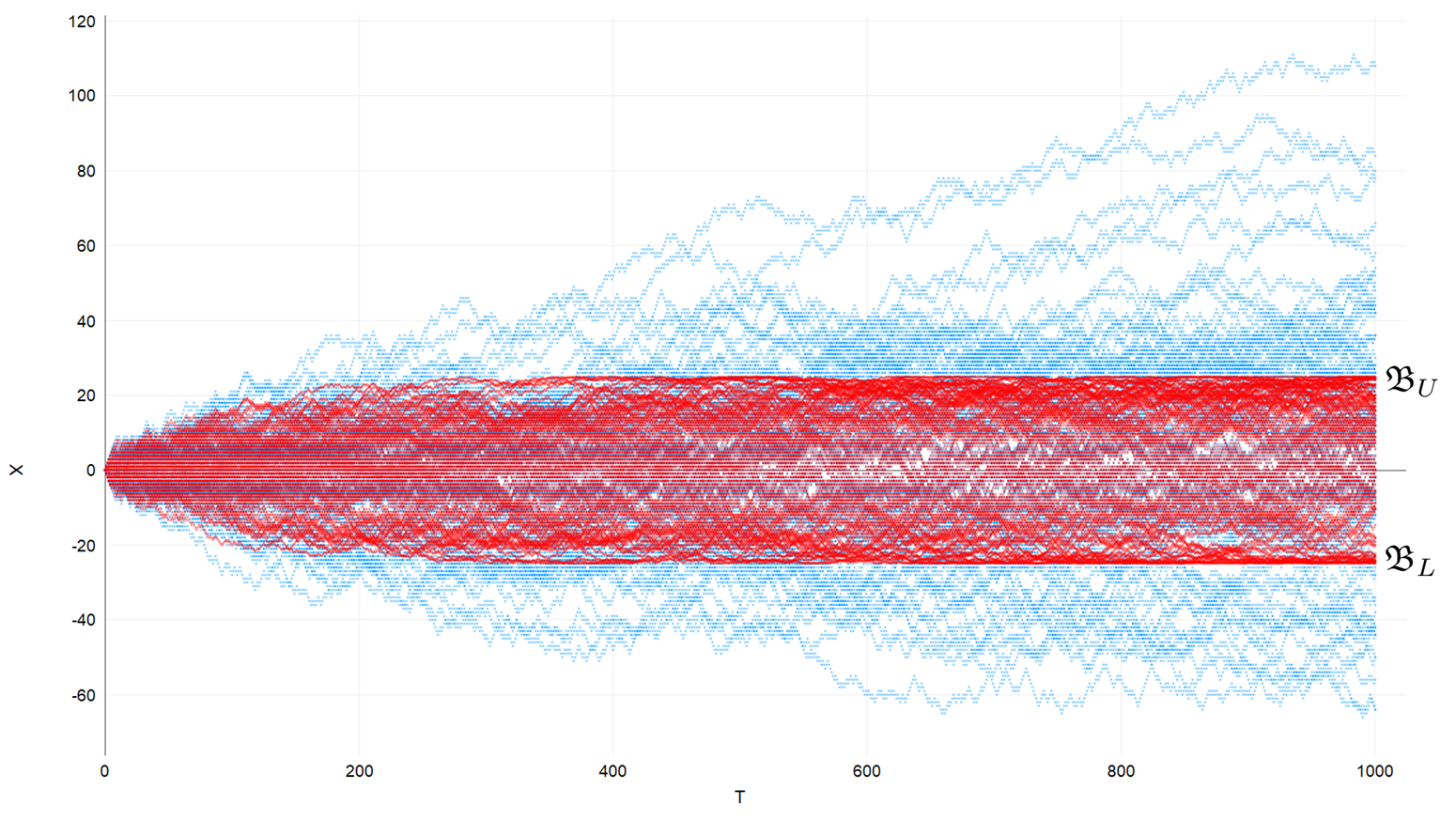}
 \caption{Hidden Reflective Barriers due to Ideal $\Psi(X_t, t)$}
  \label{Fig:SamplePositiveGrowthPathofGridTrader3of3}
\flushleft
   \textbf{\footnotesize \noindent{Blue = Unconstrained It\^{o} process,  Red = BGC It\^{o} process.
Even though there are no hard reflective barriers present, the function $\boldsymbol{\Psi(X_t, t)}$ constrains the It\^{o} process `as if' there are two hidden reflective barriers $\boldsymbol{\mathfrak{B}_L}$, $\boldsymbol{\mathfrak{B}_U}$.
}}
\end{figure}
\FloatBarrier

\bigskip \noindent
An even more specialized case that is of great interest, due to its simplicity, is where drift function $f(x): \mathbb{R} \rightarrow \mathbb{R}$ is set to $f(x)=\mu$, $\forall x \in \mathbb{R}$ and similarly, the diffusion function $g(x): \mathbb{R} \rightarrow \mathbb{R}$ is set to $g(x)=\sigma$, $\forall x \in \mathbb{R}$.
One can even define $f(x)$ and $g(x)$ in a more gradual manner such that in the limit they approach the typical constant expressions for the drift and diffusion coefficients,

\bigskip 
\begin{equation}\label{limits}
   \lim_{x \rightarrow \infty} f(x) = \mu \quad , \quad \lim_{x \rightarrow \infty} g(x) = \sigma.
\end{equation}

\bigskip \noindent
Depending on whether the generalized $f(x)$ and $g(x)$ are used, or whether the simplified $\mu$ and $\sigma$ are used, then the resulting reflective boundary theorems will either have more complexity, or less complexity, respectively.

\bigskip \noindent
Before proceeding to the Methodology section, (\ref{Eq:BGC}) is discussed within a multi-Dimensional context with some examples, which will help explain the multidimensionality of M-SBMs.

\bigskip
\begin{remark}
By multi-Dimensional diffusion, we are not referring to the usage in papers such as \citet{SacerdoteTamborrinoZucca2016} in which each dimension is reserved for each possible path or simulation, as shown in Figure \ref{Fig:MultiDimensional1}(a).

\bigskip \noindent
\begin{figure}[htb]
  \centering
   \includegraphics[width=\linewidth / 2 - \linewidth/60]{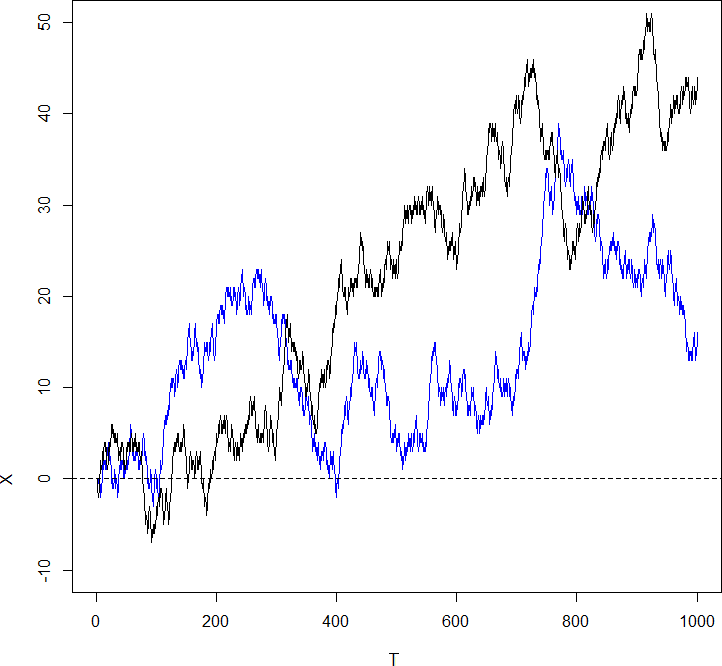} $\text{ }$
   \includegraphics[width=\linewidth / 2 - \linewidth/60]{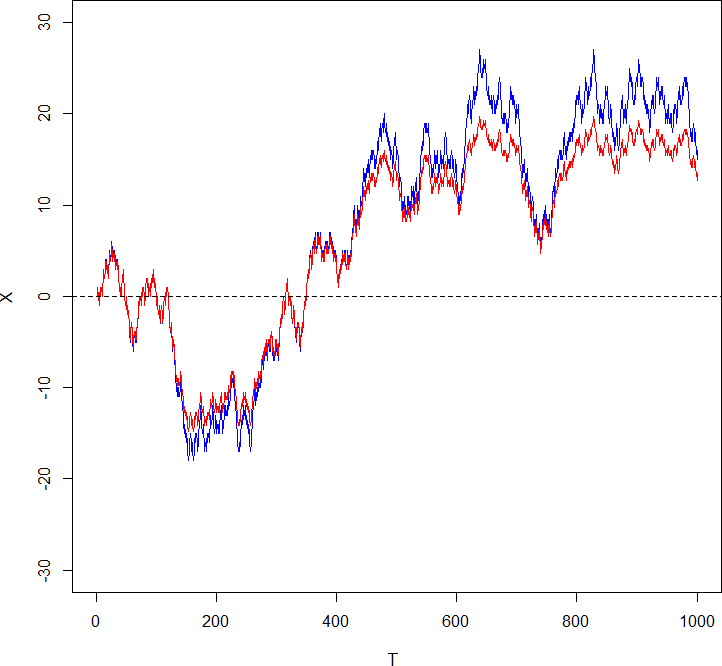}\\
   \textbf{\footnotesize \noindent{$\quad \quad$ (a). $\quad \quad \quad \quad \quad \quad \quad \quad \quad \quad \quad \quad \quad \quad \quad \quad$ (b).}}
 \caption{Clarifying Subtle Differences for BGC Stochastic Processes}
  \label{Fig:MultiDimensional1}
\flushleft
   \textbf{\footnotesize \noindent{(a). Two unconstrained 1-D It\^{o} diffusions are not considered in this paper as one 2-D It\^{o} diffusion.\\
(b). Blue = Unconstrained It\^{o} diffusion, Red = BGC It\^{o} diffusion, noting that the BGC process exhibits a `skew' above and below and can be considered as a 2-SBM.
}}
\end{figure}
\FloatBarrier

\bigskip \noindent
Notice that in Figure \ref{Fig:MultiDimensional1}(b), the BGCSP exhibits a `skew' and can be considered as a 2-SBM -which will be elaborated in the Methodology section.
Instead of the usage in Figure \ref{Fig:MultiDimensional1}(a), we use the standard interpretation and generally accepted usage of the term `multi-Dimensional diffusion' in which each dimension is reserved for each co-ordinate of the multivariate It\^{o} process, as shown in Figure \ref{Fig:MultiDimensional2}.

\begin{figure}[htb]
  \centering
   \includegraphics[width=\linewidth / 2 - \linewidth/60]{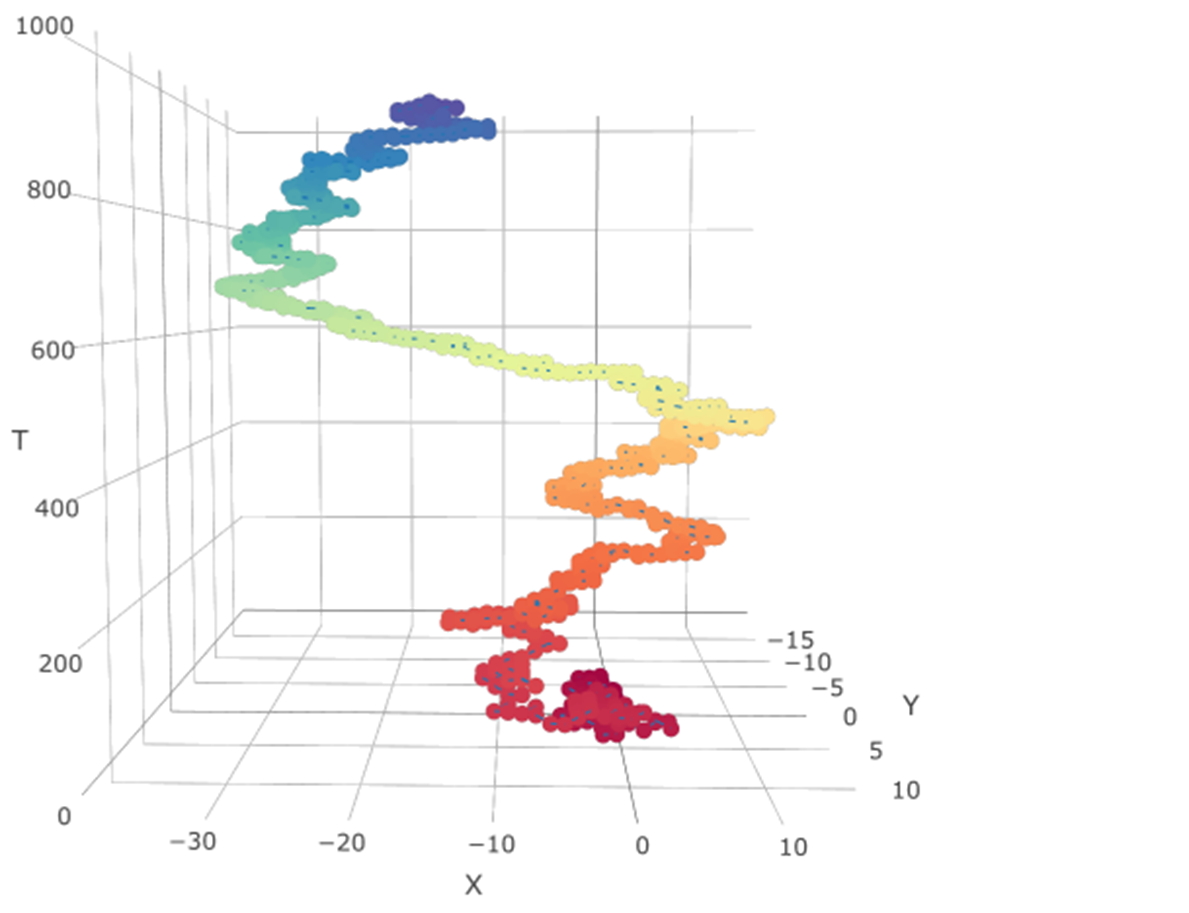}
   \includegraphics[width=\linewidth / 2 - \linewidth/60]{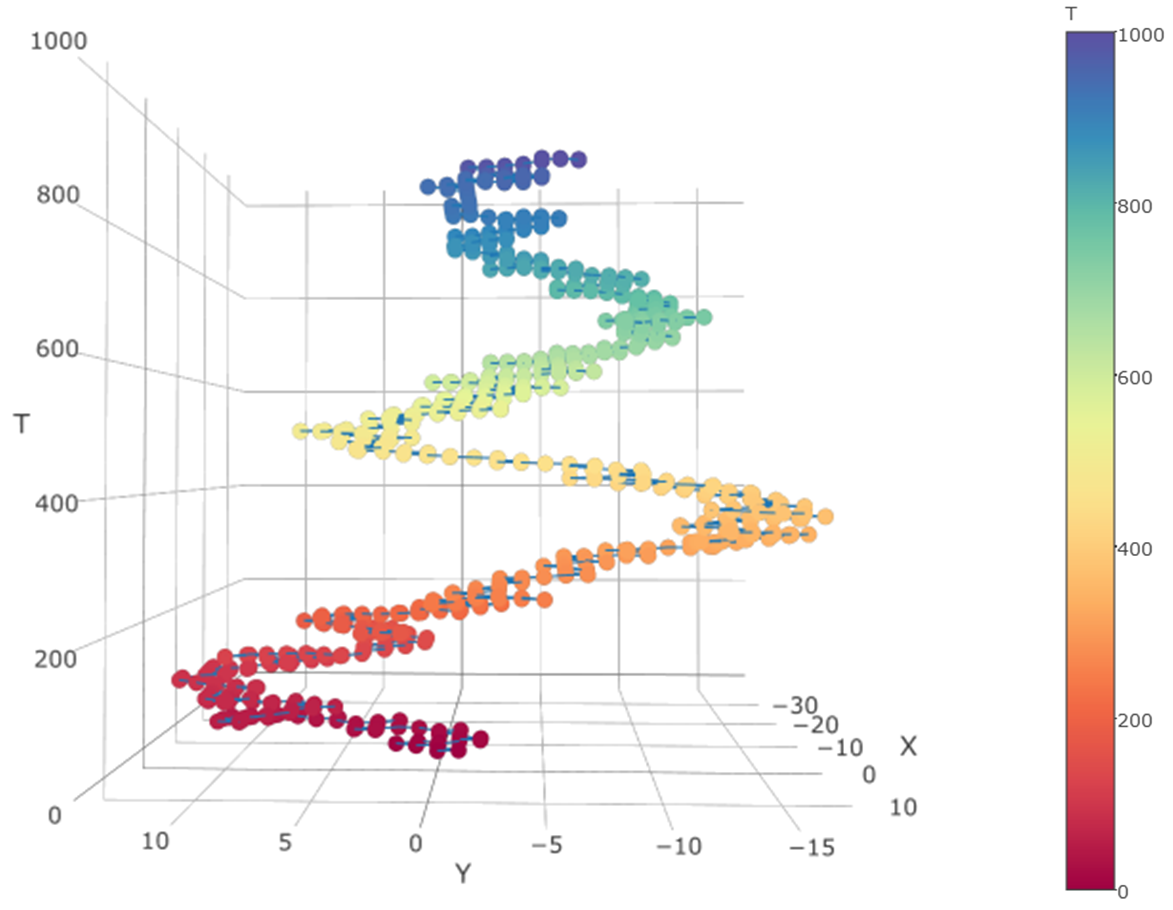}
   \textbf{\footnotesize \noindent{(a). 2-D It\^{o} Diffusion as seen from $\boldsymbol{X}$ $\quad$ (b). 2-D It\^{o} Diffusion as seen from $\boldsymbol{Y}$}}
 \caption{2-Dimensional It\^{o} Diffusion as seen from 2 Different Dimensions}
  \label{Fig:MultiDimensional2}
\flushleft
   \textbf{\footnotesize \noindent{The colour scale is a hot to cold measure indexed on the time variable $\boldsymbol{t}$ in dimension $\boldsymbol{T}$ for $\boldsymbol{t\in[0, 1000]}$.
}}
\end{figure}

\bigskip \noindent
We also note that the Methodology, Results and Discussion sections are either directly multi-Dimensional or can be extended to multi-Dimensional expressions.
\flushright    $\blacksquare$
\end{remark}
\FloatBarrier

\bigskip \noindent
We now have a more geometric understanding of how BGCSPs can be constrained along multiple dimensions, to be in a better position to express (\ref{Eq:BGC}) in a more generalized $n$-Dimensional framework in (\ref{1_0}).

\bigskip \noindent
\begin{definition}\textbf{(Definition IV - Multi-Dimensional BGC Stochastic Processes)}.
Let $\textstyle X:[0,\infty )\times \Omega \rightarrow \mathbb{R} ^{n}$ defined on a probability space $(\Omega ,\mathcal{F}, \mathcal{ \{ F \} }_{t \geq 0}, \mathbb{P})$ be an It\^{o} diffusion satisfying the conditions given in the definition of the 1-Dimensional It\^{o} process for each $\{ 1 \leq i \leq n \}$, $\{ 1 \leq j \leq m \}$, then we can form $n$ 1-Dimensional It\^{o} processes in an SDE of the form,

\bigskip
\begin{equation}\label{1_0}
\resizebox{0.9\hsize}{!}{$
   \begin{array}{ccc}
      dX_1  &    =      &   f_1 (X_{t}, t) \, d t + \underbrace{g_{1,1} (X_{t}, t) \, d W_1 (t) + \cdots +  g_{1,m} (X_{t}, t) \, d W_m (t)}_{\textbf{m}}\\
               &             &     - \underbrace{ \sgn[X_t] \, d \Psi_1 (X_t, t) - ... - \sgn[X_t] \, d \Psi_m (X_t, t)}_{\textbf{m BGC}}\\
  \vdots    &  \vdots  &    \vdots                                                    \\
      dX_n  &    =      &   f_n (X_{t}, t) \, d t + \underbrace{g_{n,1} (X_{t}, t) \, d W_1 (t) + \cdots + g_{n,m} (X_{t}, t) \, d W_m (t)}_{\textbf{m}} \\
               &            &    -  \underbrace{  \sgn[X_t] \, d \Psi_1 (X_t, t)  - ... - \sgn[X_t] \, d\Psi_m (X_t, t)}_{\textbf{m BGC}}
   \end{array}
$}
\end{equation}

\bigskip \noindent
where $W_t = (W_1 (X_t, t),..., W_m (X_t, t))$ is an $m$-Dimensional Wiener process and $f:\mathbb{R} ^{n} \rightarrow \mathbb{R} ^{n}$ and $g :\mathbb{R} ^{n} \rightarrow \mathbb{R} ^{n \times m}$ are the drift and diffusion fields respectively.
For a point $\displaystyle x \in \mathbb{R} ^{n}$, let $\displaystyle \mathbb{P} ^{x}$ denote the law of $X$ given initial datum $\displaystyle X_{0}=x$, and let $\displaystyle \mathbb{E} ^{x}$ denote expectation with respect to $\displaystyle \mathbb{P} ^{x}$.
Now, (\ref{1_0}) can also be expressed in matrix notation as,

\bigskip
\begin{equation}\label{ItoJumpDiffusion}
  d\mathbf{X}_t = \boldsymbol{f}_t \, dt + \boldsymbol{g}_t \, d\mathbf{W}_t - \boldsymbol{h}_t,
\end{equation}

\bigskip \noindent
where,

\bigskip

\[
\resizebox{0.9\hsize}{!}{$
\mathbf{X}_t = \left(
   \begin{array}{c}
       dX_1 (t)                     \\
      \vdots                              \\
       dX_n (t)
   \end{array}
\right),
\text{ }
\boldsymbol{f}_t = \left(
   \begin{array}{c}
       f_1 (X_t, t) \\
      \vdots                              \\
       f_n (X_t, t)
   \end{array}
\right),
\text{ }
\boldsymbol{g}_t = \left(
   \begin{array}{ccc}
      g_{1,1}(X_1 (t), t)  &  \cdots   &  g_{1,m}(X_1 (t), t)    \\
       \vdots    &  \ddots  &  \vdots        \\
      g_{n,1}(X_n (t), t)  &  \cdots   &  g_{n,m}(X_n (t), t)
   \end{array}
\right),
$}
\]

\[
\resizebox{0.65\hsize}{!}{$
d\mathbf{W}_t = \left(
   \begin{array}{c}
      dW_1 (t) \\
      \vdots                              \\
      dW_n (t)
   \end{array}
\right),
\text{ }
 \boldsymbol{h}_t = \left(
   \begin{array}{c}
       \sgn[X_1 (t)] \, d \Psi_1 (X_1 (t), t) \\
      \vdots                              \\
       \sgn[X_n (t)] \, d \Psi_n (X_n (t), t)
   \end{array}
\right),
$}
\]

\bigskip \noindent
for vectors $\boldsymbol{f}_{t}$, $\boldsymbol{h}_{t}$, $\boldsymbol{W}_{t}$ and matrix $\boldsymbol{g}_t$.
\hfill    $\blacksquare$
\end{definition}

\bigskip \noindent
Having reviewed the multi-Dimensional nature of unconstrained and BGC It\^{o} diffusions, the paper is structured as follows;
Section 2 provides the Literature Review, Section 3 the Methodology, Section 4 the Results and Discussion, and finally, Section 5 the Conclusion.

\bigskip
\section{Literature Review}

\subsection{Constraining Stochastic Processes by Reflective Barriers}
The constraining of stochastic processes in the form of discrete random walks and continuous Wiener processes has been researched for many decades.
By reviewing \citet{Weesakul1961} and the references therein, we see an established and rigorous analysis of random walks between a reflecting and an absorbing barrier.
\citet{Lehner1963} extended this to 1-Dimensional random walks with a partially reflecting (semipermeable) barrier.
\citet{Gupta1966} generalised this concept further to random walks in the presence of a multiple function barrier (MFB) where the barrier can either be partially reflective, absorptive, transmissive or hold for a moment, but not terminating or killing the random variable.
\citet{DuaKhadilkarSen1976} extended the work of \cite{Lehner1963} to random walks in the presence of partially reflecting barriers in which the probability of a random variable or datum reaching certain states was determined.
\citet{LionsSznitman1984} extended the research on reflecting boundary conditions through the refinement to SDEs.
\citet{Percus1985} considered an asymmetric random walk, with one or two boundaries, on a 1-Dimensional lattice.
At the boundaries, the walker is either absorbed or reflected back to the system.
\citet{BudhirajaDupuis2003} considered the large deviation properties of the empirical measure for 1-Dimensional constrained processes, such as reflecting Wiener processes, the M/M/1 queue, and discrete-time analogs.
\citet{Lepingle2009} examined stochastic variational inequalities to provide a unified treatment for SDEs existing in a closed domain with normal reflection and (or) singular repellent drift.
When the domain is a polyhedron, he proved that the reflected-repelled Wiener process does not hit the non-smooth part of the boundary.
\citet{BramsonDaiHarrison2010} examined the positive recurrence (to the origin) of reflecting Wiener processes in 3-Dimensional space.
\citet{BallRoma1998} examined the detection of mean reversion within reflecting barriers with an application to the European exchange rate mechanism (EERM).

\bigskip
\subsection{Multi-Skew Brownian Motion} The concept of skew Brownian motion (SBM) was first introduced in the book by \citet{ItoMcKean1965} as a diffusion with a drift represented by a generalized function, which solves an SDE involving its symmetric local time.
Specifically, an SBM $X = \{ X_t \}_{t \in [0,T]}$ is the solution of,

\[
X_t  = B_t + (2\alpha -1) L^{0}_{t} (X), \quad \alpha \in (0,1),
\]

\bigskip \noindent
where $B = \{ B_t \}_{t \in [0,T]}$ is a standard Brownian motion (BM).
However, standard Brownian motion is oftentimes abbreviated as SBM, so to reduce any possible confusion and to eliminate any reference to the original botanical context of Robert Brown -to which the term `Brownian motion' is attributed -we will express $B_t$ in the rest of this paper in the more mathematically precise context as a Wiener process $W = \{ W_t \}_{t \in [0,T]}$.
$L^{0}_{t}(X) = \lim_{\epsilon \downarrow 0} \frac{1}{2 \epsilon} \int^{t}_{0} \mathds{1}_{ \{ |B_{X_s}| \leq \epsilon \} } \, ds$ is the symmetric local time at $X_0$.
Note that for $\alpha =\frac{1}{2}$, the above equation is reduced to a Wiener process. 
\citet{HarrisonShepp1981} then considered diffusions with a discontinuous local time.
The literature on SBMs was consolidated by \citet{HarrisonShepp1981} and later by \citet{Lejay2006}.
Applications of SMBs were extended by \citet{Ramirez2011} by applying multi-SBM (M-SBM) and diffusions in layered media that involve advection flows.
\citet{AppuhamillageSheldon2012} linked SBMs to existing research by deriving the first passage time (FPT) of SBM.
In 2015, the multiple barrier research of \cite{Ramirez2011} was extended by \citet{AtarBudhiraja2015}, \citet{OuknineRussoTrutnau2015} who collapsed barriers to an accumulation point, and by \citet{DereudreMazzonettoRoelly2015} who derived an explicit representation of the transition densities of SBM with drift and two semipermeable barriers.
\citet{Mazzonetto2016} extended her prior research \cite{DereudreMazzonettoRoelly2015} on SBMs by deriving exact simulations of SBMs and M-SBMs with discontinuous drift in her Doctoral dissertation.
\citet{GairatShcherbakov2017} applied SBMs and their functionals to finance.
\citet{Krykun2017} also extended the convergence of SBM with local times at several points that are contracted into a single one.
\citet{Mazzonetto2019} has also recently examined the rates of convergence to the local time of oscillating and SBMs.

\bigskip \noindent
For applications of BGC stochastic processes, the reader is referred to \cite{TarantoKhan2020_4}, \cite{TarantoKhan2020_5}, \cite{TarantoKhan2020_6}, \cite{TarantoKhan2020_7}.

\bigskip
\section{Methodology}

\noindent
Before proceeding to the main result of this paper, it is instructive to establish a theoretical foundation by considering the key research for It\^{o} diffusions constrained by two reflective barriers and then examining the necessary extensions that need to be derived for M-SBM constrained It\^{o} diffusions.

\bigskip 
\subsection{It\^{o} Diffusions Constrained by Two Reflective Barriers}
Given a filtered probability space $\Lambda := (\Omega, \mathcal{F}, \{ \mathcal{F}_t \}_{t \geq 0}, \mathbb{P})$ with the filtration $\{ \mathcal{F}_t \}_{t \geq 0}$, then the reflected diffusion $\{ X_t : t \geq 0 \}$ with two-sided barriers $\mathfrak{B}_L$, $\mathfrak{B}_U$ at $a$, $b$ respectively can be defined as,

\begin{equation}\label{Eq:3alpha}
dX_t = f (X_t) \, dt +g(X_t) \, dW_t + \underbrace{ \overbrace{d \mathcal{A}_t}^{a} - \overbrace{d \mathcal{B} _t}^{b} }_{\text{regulators}}, \quad X_0 \in (a,b), \quad x \in [a,b].
\end{equation}

\bigskip \noindent
\begin{remark}
The process $\mathcal{A}$ and $\mathcal{B}$ are known in the literature as `regulators' for the points $a$ and $b$, however, we believe that a better term is `detectors' because they mainly detect or count how many times $X_t$ reaches $a$ and $b$.
\end{remark}

\bigskip \noindent
Here, the drift $f(x)$ is Lipschitz continuous, the diffusion $g(x)$ is strictly positive and Lipschitz continuous.
$a$, $b$ with $-\infty < a < b < + \infty$ are given real numbers, and $(W_t, 0 \leq t < \infty)$ is the 1-Dimensional standard Wiener process on $\Lambda$.
The processes $\mathcal{A} = \{ \mathcal{A}_t \}_{t \geq 0}$ and $\mathcal{B} = \{ \mathcal{B}_t \}_{t \geq 0}$ are the minimal non-decreasing and non-negative processes, which restrict the process $X_t \in [a,b], \forall t \geq 0$.
More precisely, the processes $\{ \mathcal{A}_t, t \geq 0 \}$ and $\{ \mathcal{B}_t, t \geq 0 \}$ increase only when $X_t$ hits the boundary $a$ and $b$, respectively, so that $\mathcal{A}_0 = \mathcal{B}_0 = 0$, $\mathds{1}$ is the characteristic function of the set and,

\begin{equation}\label{1_1}
   \int^{\infty}_{0} \mathds{1}_{\{ X_t > a \} } \, d \mathcal{A}_t = 0, \quad \int^{\infty}_{0} \mathds{1}_{\{ X_t < b \} } \, d \mathcal{B}_t = 0.
\end{equation}

\bigskip \noindent
Furthermore, the processes $\mathcal{A}$ and $\mathcal{B}$ are uniquely determined by the following properties (\citet{Harrison1986}),

\begin{enumerate}
\item both $t \rightarrow \mathcal{A}_t$ and $t \rightarrow \mathcal{B}_t$ are nondecreasing processes,
\item $\mathcal{A}$ and $\mathcal{B}$ increase only when $X=a$ and $X=b$, respectively, that is $\int^{t}_{0} \mathds{1}_{\{ X_s = a \} } \, d \mathcal{A}_s = \mathcal{A}_t$ and $\int^{t}_{0} \mathds{1}_{\{ X_s = b \} } \, d \mathcal{B}_s = \mathcal{B}_t$, for $t \geq 0$.
\end{enumerate}

\bigskip \noindent
We can consider the two reflective barriers as forming a 2-SBM.
Furthermore, it is instructive for BGCSP to see the two barriers $a$ and $b$ in $\mathbb{R}^2$, shown in Figure \ref{Fig:DiffusionBetweenTwoConstantReflectiveBarriers}(a) as embedded in $\mathbb{R}^3$ by a governing BGC surface $\Psi (X_t, t)$, as shown in Figure \ref{Fig:DiffusionBetweenTwoConstantReflectiveBarriers}(b).

\begin{figure}[ht]
   \centering
 \includegraphics[width=\linewidth]{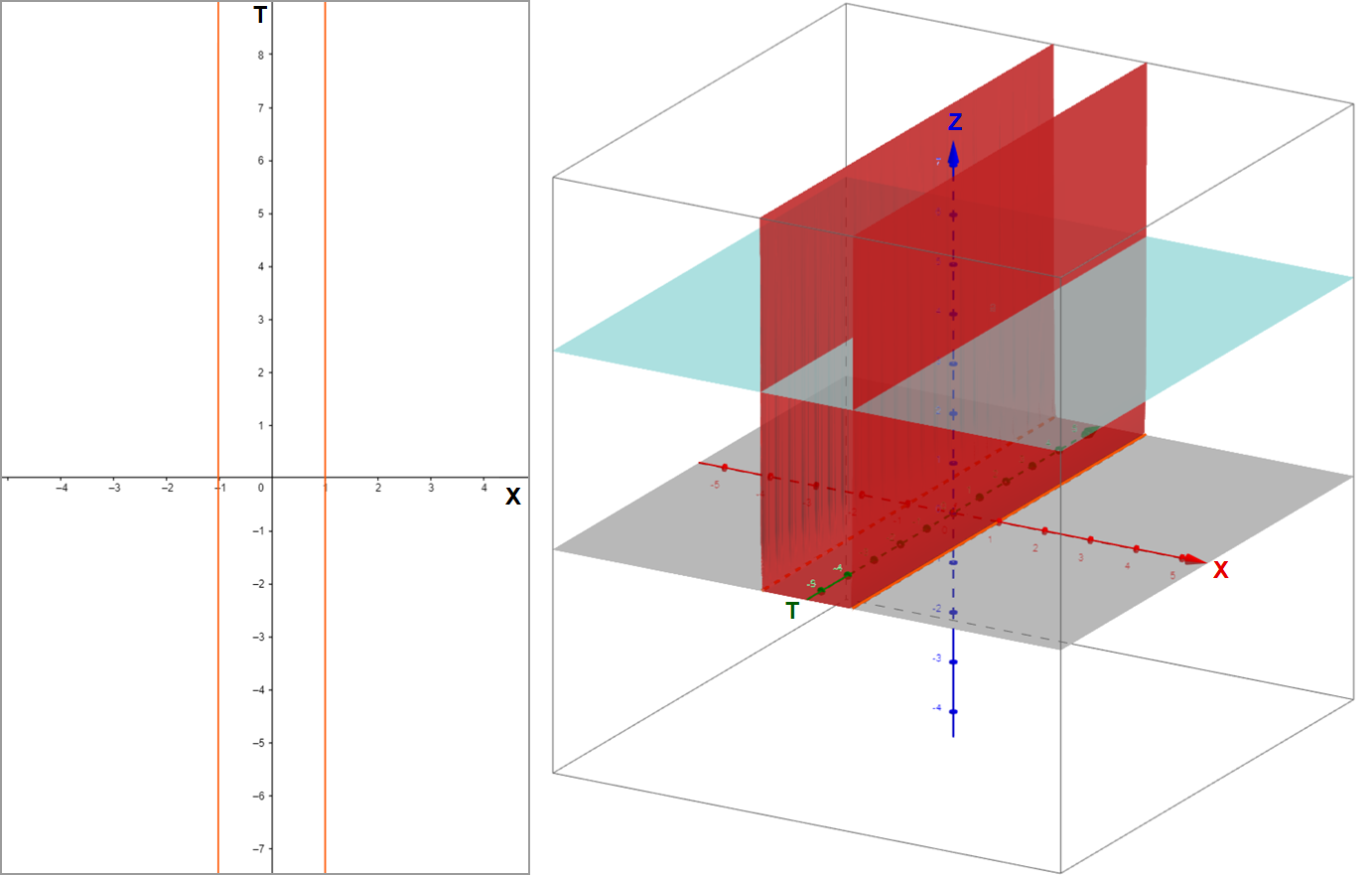}
   \textbf{\footnotesize \noindent{
(a). Barriers in $\boldsymbol{\mathbb{R}^2}$ - Contour plot $\quad \quad$ (b). Barriers in $\boldsymbol{\mathbb{R}^3}$- Contour \& surface plot}}
   \caption{Diffusion Between Two Constant Reflective Barriers}
   \label{Fig:DiffusionBetweenTwoConstantReflectiveBarriers}
\flushleft
   \textbf{\footnotesize \noindent{As the light blue plane $\boldsymbol{z = k}$ for $\boldsymbol{k \in \mathbb{R}_{+}}$ descends to the origin, the orange contour lines of the red constant reflective barriers do not change.
The contour lines arise from when the blue plane intersects the barrier surface in $\mathbb{R}^3$ that acts as two barriers in $\mathbb{R}^2$.
}}
\end{figure}

\bigskip
\begin{remark}
Many papers such as \cite{HuWangYang2012} and \citet{Linetsky2005} define the two barriers at the boundaries of $[0, r]$ for some $r \in \mathbb{R}_{+}$.
By applying a series of transformations, one can map their findings to the context of BGCSP, as shown in Figure \ref{Fig:MappingFunction}.
\end{remark}

\bigskip

\begin{figure}[htb]
  \centering
   \includegraphics[width=\linewidth]{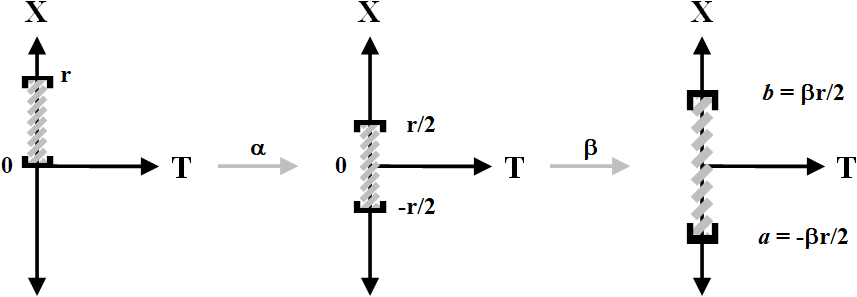}
 \caption{Mapping Traditional Barriers to BGC Barriers}
  \label{Fig:MappingFunction}
\flushleft
\textbf{\footnotesize \noindent{By applying the transformations $\boldsymbol{\alpha}$ and $\boldsymbol{\beta}$ to the barriers at $\boldsymbol{0}$ and $\boldsymbol{r}$, the result is a linear combination of mappings with the same properties as the original for the barriers at $\boldsymbol{a}$ and $\boldsymbol{b}$.
}}
\end{figure}
\hfill    $\blacksquare$

\bigskip
\subsection{Multi-Skew Brownian Motion}

Having examined It\^{o} diffusions constrained by two reflective barriers, we now consider the so-called multi-skew Brownian motion, constrained by multiple barriers of varying degrees of reflectiveness.

\bigskip
\begin{definition} \label{Def:3_3}
\textbf{(M-SBM)}. A multi-skew Brownian motion (M-SBM) represented (adapted from \citet{Mazzonetto2016}) by $(\beta_1, ..., \beta_n)$-SBM, or more simply by $\beta$-SBM with $n$ semipermeable barriers of varying permeability coefficients, respectively $\beta = (\beta_1, ..., \beta_n)$, $x_0$ is the starting position, the coefficients $\beta_j \in [-1,1]$, barriers $x_1 <... < x_n$, local times $L^{x_j }_{t}$, and $\mathcal{E}$ is the set of all parameters of the M-SBM, then the M-SBM is expressed as,

\bigskip
\begin{eqnarray*}
 & & {
\begin{cases}
   \begin{array}{ccl}
\displaystyle  dX_t           & = & \mu \, dt + \sigma \, d W_t + \beta_1 \, dL^{x_1}_{t} + ...+ \beta_n \, dL^{x_n}_{t} \\
\displaystyle  X_0            & = & x_0  \\
\displaystyle  \mathcal{E} & = & \big\{ \mu, \sigma, \left( \beta_1, ..., \beta_n \right) \big\} \in \mathbb{R}  \\
\displaystyle L^{x_j}_{t}   & = & \int ^{t}_{0} \mathds{1}_{ \{ X_s = x_j \} } \, d L^{x_j}_{s},  \text{ }    \forall j \in \{ 1, ..., n \}
   \end{array}.
\end{cases}
}
\end{eqnarray*}       \hfill    $\blacksquare$
\end{definition}

\bigskip
\begin{remark}
The term $\sigma$ has been added to the \cite{Mazzonetto2016} definition so that the process can fit a wider range of models.
We require that $\beta_i \in [-1, 1]$.
The cases when $\beta_i = \{ -1, 1 \}$ are said to exhibit zero permeability (i.e. impermeability or full reflectiveness), and when $\beta_i \in (-1, 1)$ the process is said to exhibit partial reflectiveness (i.e. semi permeability).
Note that a 0-SBM is simply a Wiener process and a $\pm$ $1$-SBM is a positively/negatively reflected Wiener process.
The definition of M-SBM is illustrated in Figure \ref{Fig:StandardFramework}, representing a typical example of M-SBM.
The standard definition of a skew Brownian motion has a drift term $\mu \in \mathbb{R}$ making it no longer, strictly speaking, Brownian motion.
\end{remark}

\begin{figure}[ht]
   \centering
 \includegraphics[width=\linewidth*3/4]{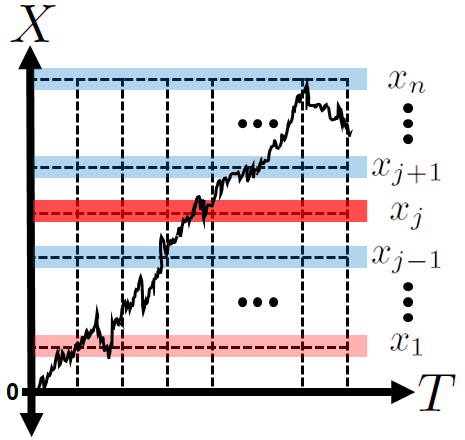}
   \caption{Example Standard M-SBM Framework}
   \label{Fig:StandardFramework}
\flushleft
   \textbf{\footnotesize \noindent{(a). Red barriers have permeability values $\boldsymbol{\beta_j < 0}$ so reflect to the left (upwards). The more negative the value is within $\boldsymbol{\beta_j \in [-1, 0)}$, the more reflective the barrier is (i.e. the less permeable it is).\\
(b). Blue barriers have permeability values $\boldsymbol{\beta_j > 0}$ so reflect to the right (downwards).
The more positive the value is within $\boldsymbol{\beta_j \in (0, 1]}$, the more reflective the barrier (i.e. the less permeable it is).\\
(c). Barriers that have permeability values $\boldsymbol{\beta_j = 0}$ are not depicted since they have no (constraining) effect.
}}
\end{figure}
\FloatBarrier

\bigskip \noindent
The M-SBM of Definition \ref{Def:3_3} allows any barrier combination to be either fully reflective or semipermeable.

\bigskip
\begin{remark}
If the permeability is $\beta_j =1$ at the barrier $x_j$ for some $j \in \{ 1,...,j_1,j,j+1,...,n\}$ and the initial position is $x_0 \in (x_j, +\infty)$, then the lower barriers $x_1,...,x_{j-1}$ will almost surely be never reached \cite{Mazzonetto2016}.
For this to happen, it must be that $\beta_j = -1$, so that as the It\^{o} diffusion descends (down) to $x_j$, it is fully reflected back (up).
This of course assumes that the diffusion coefficient $\sigma$ is `relatively small' and so allows the It\^{o} diffusion to be `well behaved' and never `jump over' and go below the $x_j$ barrier, as also illustrated in Figure \ref{Fig:StandardFramework}.
\end{remark}

\bigskip \noindent
\begin{theorem}\label{Thm:3_7}
\textbf{(Multiple Barriers of M-SBM Merging to One, adapted from \citet{Mazzonetto2016})}.
Before expressing the skewness parameter $\beta$ for a general number of barriers $n$, we derive $\beta$ for the first two simplest scenarios.

\bigskip \noindent
If $n=2$, $\beta_1, \beta_2 \in [-1,1]$, $\mu \in \mathbb{R}$ and $x^{(n)}_{2}=x_1 + \frac{1}{n}$, $\forall n \in \mathbb{N}$.
Let,

\[
 \beta := \frac{\beta_1 + \beta_2}{1 + \beta_1 \beta_2}.
\]

%

\bigskip \noindent
Let us denote by $(X^{(n)}_{t})_t$ the $(\beta_1, \beta_2)$-SBM with drift $\mu$, barriers $x_1, ..., x_n$, and denote by $(Y_t)_t$ the 1-SBM with drift $\mu$, and barrier $x_1$.
Let us assume $X^{(n)}_0 = Y_0$, then $X^{(n)}$ converges to $Y$ in the following sense,

\[
 \mathbb{E} \left[ \sup_{s \in [0,t]} |X^{(m)}_s - Y_s| \right] \xrightarrow{m \to \infty} 0, \text{ } \forall t \geq 0.
\]

\bigskip \noindent
The same holds in the case of $n > 2$ barriers merging.
In this case $(X^{(n)}_t )_t$ is the $(\beta_1,...,\beta_n)$-SBM with drift $\mu \in \mathbb{R}$, skewness parameters
$\beta_1,...,\beta_n \in [-1, 1]$ and barriers $x_1 \in \mathbb{R}$, $x_{j+1} := \frac{j}{n} + x_1, \text{ } \forall j \in \{ 1,...,n-1 \}$.

\bigskip \noindent
The skewness parameter $\beta$ of the limit $1$-SBM is given by,

\begin{equation}\label{Eq:Beta}
\beta := \frac{\displaystyle \prod^{n}_{j=1} (1+\beta_j) - \prod^{n}_{j=1} (1-\beta_j)}
                   {\displaystyle \prod^{n}_{j=1} (1+\beta_j) + \prod^{n}_{j=1} (1-\beta_j)}.
\end{equation}

\bigskip \noindent
If $n$ is even,

\begin{equation}\label{Eq:Even}
\beta = \frac{\displaystyle \sum^{n}_{j=1} \beta_j + \sum_{j_1 < j_2 < j_3} \beta_{j_1}\beta_{j_2}\beta_{j_3} +...+ \sum_{j_1 < ... < j_{n-1}} \beta_{j_1}...\beta_{j_{n-1}}}
                  {\displaystyle 1 + \sum_{j_1 < j_2} \beta_{j_1}\beta_{j_2} + \sum_{j_1 < ... < j_4} \beta_{j_1}\beta_{j_2}\beta_{j_3}\beta_{j_4} +...+ \beta_{1}\beta_{2}...\beta_{j_n}}.
\end{equation}

\bigskip \noindent
If $n$ is odd,

\begin{equation}\label{Eq:Odd}
\beta = \frac{\displaystyle \sum^{n}_{j=1} \beta_j + \sum_{j_1 < j_2 < j_3} \beta_{j_1}\beta_{j_2}\beta_{j_3} +...+ \beta_{j_1}...\beta_{j_n}}
                  {\displaystyle 1 + \sum_{j_1 < j_2} \beta_{j_1}\beta_{j_2} + \sum_{j_1 < ... < j_4} \beta_{j_1}\beta_{j_2}\beta_{j_3}\beta_{j_4} +...+ \sum_{j_1<...<j_{n-1} }\beta_{1}\beta_{2}...\beta_{j_{n-1} }}.
\end{equation}

\end{theorem}

\begin{proof}
Refer to \cite{Mazzonetto2016} and \cite{LeGall1984}.
\end{proof}

\bigskip \noindent
The M-SBM framework also only considers one half-plane at a time, so that the transition density (or distribution) of the upper plane is assumed to be the same for the lower half plane, which is not always the case (except for BGCSP).
We show below that whilst BGCSPs are a special case of M-SBMs, they have some unique properties that make them of particular interest among the larger class.

\bigskip
\subsection{Constructing BGC Stochastic Processes}

\noindent
We can compliment Mazzonetto by condensing all possible local barrier combinations to the following four possible global barrier combinations that comprise a lower barrier $\alpha_j$ and an upper barrier $\alpha_k$, as shown in Figure \ref{Fig:GeneralisedBarrierCombinationArguments}.

\begin{figure}[ht]
   \centering
 \includegraphics[width=\linewidth/2 - \linewidth/60]{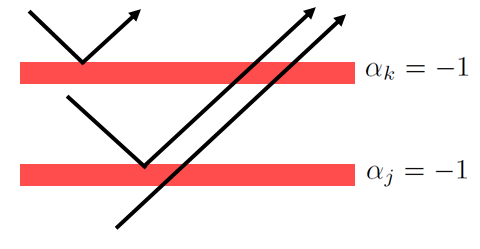}
 \includegraphics[width=\linewidth/2 - \linewidth/60]{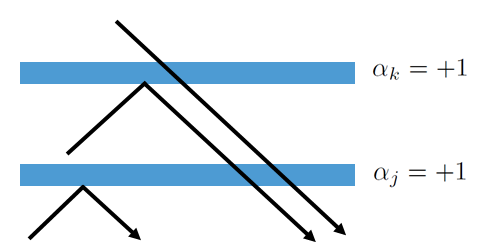}\\
\textbf{\footnotesize \noindent{(a). $\quad \quad \quad \quad \quad \quad \quad \quad \quad \quad \quad \quad \quad \quad \quad \quad$ (b). \quad \quad}}\\
 \includegraphics[width=\linewidth/2 - \linewidth/60]{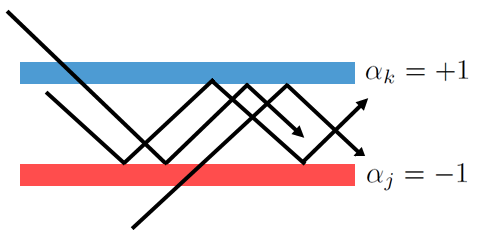}
 \includegraphics[width=\linewidth/2 - \linewidth/60]{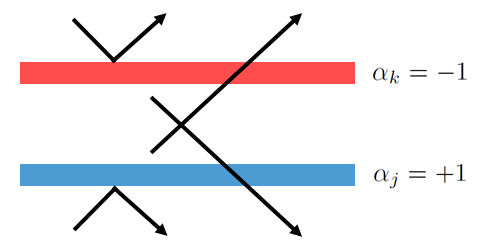}\\
\textbf{\footnotesize \noindent{(c). $\quad \quad \quad \quad \quad \quad \quad \quad \quad \quad \quad \quad \quad \quad \quad \quad$ (d). \quad \quad}}\\
  \caption{Generalised Barrier Combination Arguments}
   \label{Fig:GeneralisedBarrierCombinationArguments}
\flushleft
  \textbf{\footnotesize \noindent{(a). All It\^{o} diffusions will almost surely end up \textbf{\textit{above}} the two negative fully reflective barriers.\\
(b). All It\^{o} diffusions will almost surely end up \textbf{\textit{below}} the two positive fully reflective barriers.\\
(c). All It\^{o} diffusions will almost surely end up \textbf{\textit{within}} the two fully reflective barriers.\\
(d). All It\^{o} diffusions will almost surely end up \textbf{\textit{above or below (but not in between)}} the two fully reflective barriers.
}}
\end{figure}
\FloatBarrier

\bigskip \noindent
The diagrammatic summary of possible cases represented in Figure \ref{Fig:GeneralisedBarrierCombinationArguments} is formally stated as Lemma \ref{Lem:3_7} which is then used in Theorem \ref{SkewnessParameterofBGCStochasticProcesses}.
This Theorem formally expresses that the barriers of an M-SBM merge to a 1-SBM, to which the process converges.

\bigskip
\begin{lemma}\label{Lem:3_7}
If any $|\alpha_j|<1$ for some $|\alpha_k|=1$, or similarly for any $|\alpha_k|<1$, then the barrier $\pm 1$ will dominate the barrier $\ne \pm 1$, almost surely, as shown in Figure \ref{Fig:GeneralisedBarrierCombinationArguments}.
Furthermore, if there are more than two fully reflective barriers, they will merge and effectively operate as one of the four possible combinations of Figure \ref{Fig:GeneralisedBarrierCombinationArguments}.
\end{lemma}

\begin{proof}
We first assume that there exists only one reflective barrier $| \alpha_j |=1$ and $n$ semipermeable barriers $| \alpha_j |< 1$.
We then consider the two SDEs;

\begin{equation}\label{Eq:Difference_1}
\begin{array}{rcl}
  X_t & = & \mu_1 \, dt + \sigma_1 \, dW_t + \underbrace{ \alpha_1 \, dL^{\alpha_1}_{t} }_{| \alpha_i | = 1} + \underbrace{\alpha_2 \, dL^{\alpha_2}_{t} +...+ \alpha_n \, dL^{\alpha_n}_{t} }_{| \alpha_i | < 1}   \\
  Y_t & = & \mu_2 \, dt + \sigma_2 \, dW_t + \underbrace{ \alpha_1 \, dL^{\alpha_1}_{t} }_{| \alpha_i | = 1}
\end{array},
\end{equation}

\bigskip \noindent
where $Y_t$ is an unconstrained It\^{o} diffusion and $X_t$ is a constrained It\^{o} diffusion according to the above barrier constraints.
Let $\delta = \sum^{n}_{i=1} \alpha_i$, so $\delta < 1$ or $\delta \geq 1$.

\bigskip \noindent
If $\delta < 1$, then $\alpha_i$ will dominate $\delta$ as it will vanish (i.e. $\delta \rightarrow 0$) such that, 

\[
\displaystyle \sup_{t \rightarrow \infty} \Big\{ | X_t | - | Y_t | \Big\} = 0.
\]

\bigskip \noindent
If $\delta \geq 1$, then $\delta$ will dominate $\alpha_1$ and merge (i.e. $\alpha_1 \rightarrow \delta$) such that, 

\[
\displaystyle \sup_{t \rightarrow \infty} \Big\{ | X_t | - | Y_t | \Big\} = X_t.
\]

\bigskip \noindent
Next, assume that there exist two fully reflective barriers $| \alpha_j | = 1$, $| \alpha_k | = 1$ and $n$ semipermeable barriers $| \alpha_i | < 1$.
(\ref{Eq:Difference_1}) now equates to,

\begin{equation}\label{Eq:Difference_2}
\begin{array}{rcl}
  X_t & = & \mu_1 \, dt + \sigma_1 \, dW_t + \underbrace{ \alpha_1 \, dL^{\alpha_1}_{t} }_{| \alpha_i | = 1} + \underbrace{\alpha_2 \, dL^{\alpha_2}_{t} +...+ \alpha_n \, dL^{\alpha_n}_{t} }_{| \alpha_j | = 1, \text{ } | \alpha_k | = 1}       \\
&   &
 +  \underbrace{ \alpha_1 \, dL^{\alpha_1}_{t} }_{| \alpha_i | = 1} + \underbrace{\alpha_2 \, dL^{\alpha_2}_{t} +...+ \alpha_n \, dL^{\alpha_n}_{t} }_{\delta = | \alpha_i | < 1}                         \\
  Y_t & = & \mu_2 \, dt + \sigma_2 \, dW_t + \underbrace{ \alpha_1 \, dL^{\alpha_1}_{t} }_{\delta = | \alpha_i | < 1}
\end{array}.
\end{equation}

\bigskip \noindent
If $\delta < 1$, then $\alpha_j$ and (or) $\alpha_k$ will dominate $\delta$ and as it will vanish (i.e. $\delta \rightarrow 0$) and if $\delta \geq 1$, then $\delta$ will dominate $\alpha_j$ and (or) $\alpha_k$ hence merge to $\alpha_j$ and $\alpha_k$, such that $\sup_{t \rightarrow \infty} \Big\{ | X_t | - | Y_t | \Big\} = 0$ and $\sup_{t \rightarrow \infty} \Big\{ | X_t | - | Y_t | \Big\} = X_t$, respectively.

\bigskip \noindent
Finally, if there are more than $N \geq 3$ fully reflective barriers $| \alpha_i | = 1$ and $n$ semipermeable barriers, then the new barriers will effectively be a linear combination of any two possible combinations in Figure \ref{Fig:GeneralisedBarrierCombinationArguments}, depending on how the fully reflective barriers of $N$ are defined, completing the proof for all scenarios.
\end{proof}

\bigskip \noindent
To contrast Figure \ref{Fig:DiffusionBetweenTwoConstantReflectiveBarriers} for two reflective constant barriers, for BGCSP we have two hidden reflective barriers which also constrain the interior between the boundaries, as shown in Figure \ref{Fig:DiffusionBetweenTwoBGCReflectiveBarriers}, where (a) shows the multiple barriers in $\mathbb{R}^2$, and (b) shows how the multiple barriers are projected from $\mathbb{R}^3$.

\begin{figure}[ht]
   \centering
 \includegraphics[width=\linewidth]{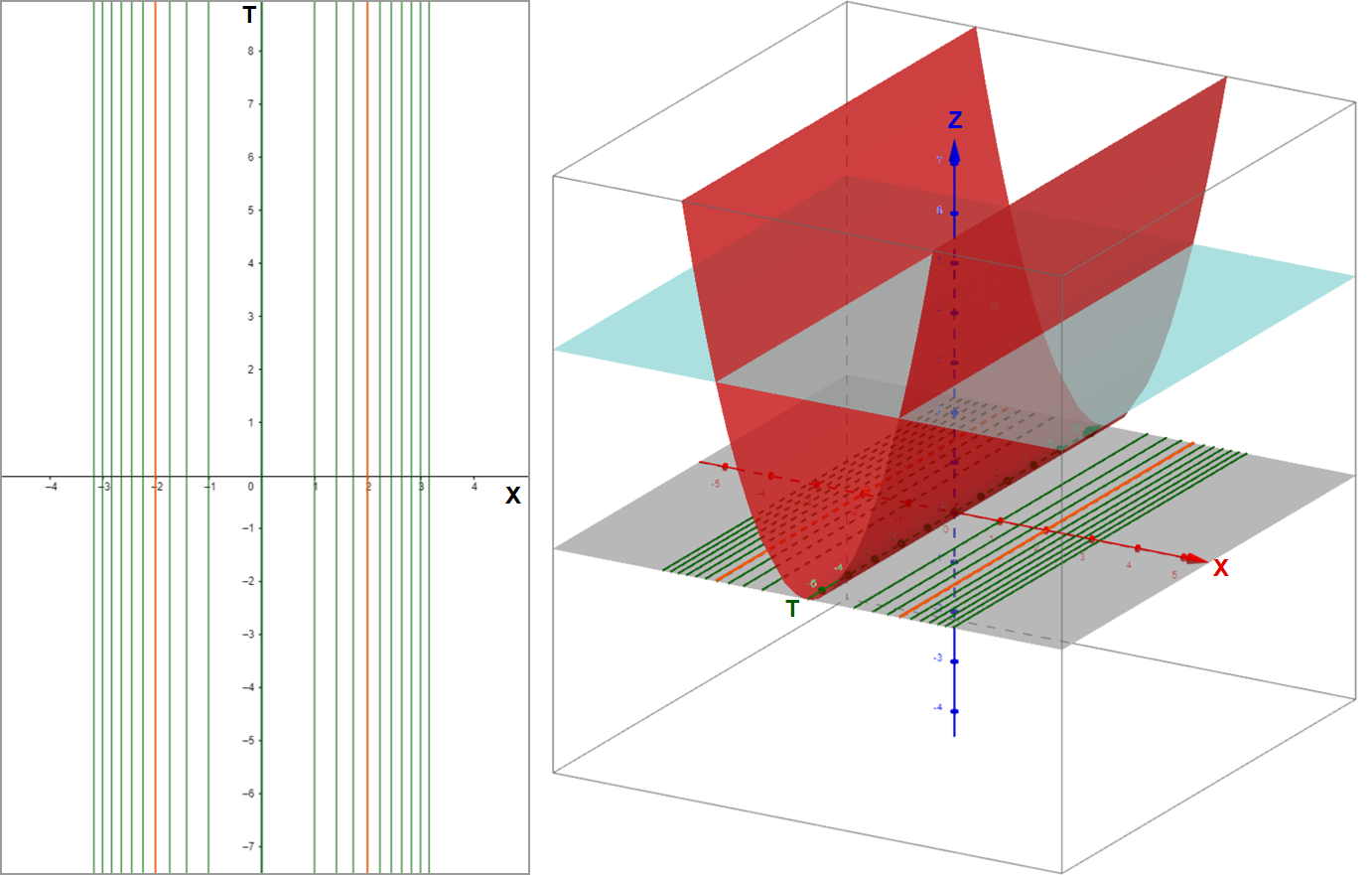}
   \textbf{\footnotesize \noindent{(a). Barrier in $\boldsymbol{\mathbb{R}^2}$ - Contour plot $\quad \quad$ (b). Barrier in $\boldsymbol{\mathbb{R}^3}$- Contour \& surface plot}}
   \caption{Diffusion Between Two BGC Reflective Barriers}
   \label{Fig:DiffusionBetweenTwoBGCReflectiveBarriers}
\flushleft
   \textbf{\footnotesize \noindent{As the light blue plane $\boldsymbol{z = k}$ for $\boldsymbol{k \in \mathbb{R}_{+}}$ descends to the origin, the green (and orange) contour lines of the BGC reflective barriers change in accordance with the red BGC function $\boldsymbol{\Psi (X_t, t)}$.
}}
\end{figure}

\bigskip \noindent
Leveraging the work of \citet{Ramirez2011}, we partition $X$ into countably infinite intervals $I_k = (x_k, x_{k+1})$, $\forall k \in \mathbb{R}$ forming the sequence $\{ I_{-\infty},...,I_{-1},I_{1},...,I_{\infty} \}$ such that the standard conditions are met; $I_k \bigcap I_{k+1} = \varnothing \text{ }\text{ } \forall k \in \mathbb{R}$, $\varnothing \notin X$ and $\bigcup^{\infty}_{k=-\infty} I_k = X$.
We wish to shrink the size of each interval $| I_k | = | x_{k+1} - x_k |$ to zero as we apply more and more intervals, where $\lim_{k \rightarrow \infty} | x_{k+1} - x_k | \rightarrow 0$ and $\int^{\infty}_{-\infty} I_k \, dk < \infty$.
This is because we wish to constrain the It\^{o} diffusion by the BGC function $\Psi (X_t, t) \in \mathbb{R}$.

\bigskip \noindent
In terms of BGC stochastic processes, we effectively have a $(\beta_{-n}, ..., \beta_{-1}, \beta_1, ..., \beta_n)$-SBM and will express it as,

\bigskip
\bigskip
\begin{eqnarray}\label{Eqn:3_9}
 & & {
\begin{cases}
   \begin{array}{ccl}
\displaystyle  dX_t           & = & \mu \, dt + \sigma \, d W_t 
+ \underbrace{\underbrace{\sum^{-1}_{j=-n} \beta_{j} \, dL^{x_{j} }_{t}}_{<0}
+ \underbrace{\sum^{n }_{j=1 } \beta_{j} \, dL^{x_{j} }_{t} }_{>0}}_{\displaystyle \Psi(X_t, t)} \\
\displaystyle X_0             & = & 0   \\
\displaystyle \mathcal{E} & = &   \big\{ \mu, \sigma, ( \beta_{-n}, ..., \beta_{-1}, \beta_{1},..., \beta_n ) \big\} \in \mathbb{R} \\
\displaystyle L^{x_j}_{t}  & = & \int ^{t}_{0} \mathds{1}_{ \{ X_s = z_j \} } \, d L^{x_j}_{s},  \text{ }    j \in \{ -n,..., -1,1, ..., n \}
   \end{array},
\end{cases}
}
\end{eqnarray}

\bigskip \noindent
as illustrated in Figure \ref{Fig:Framework}.

\begin{figure}[ht]
   \centering
 \includegraphics[width=\linewidth - \linewidth/4]{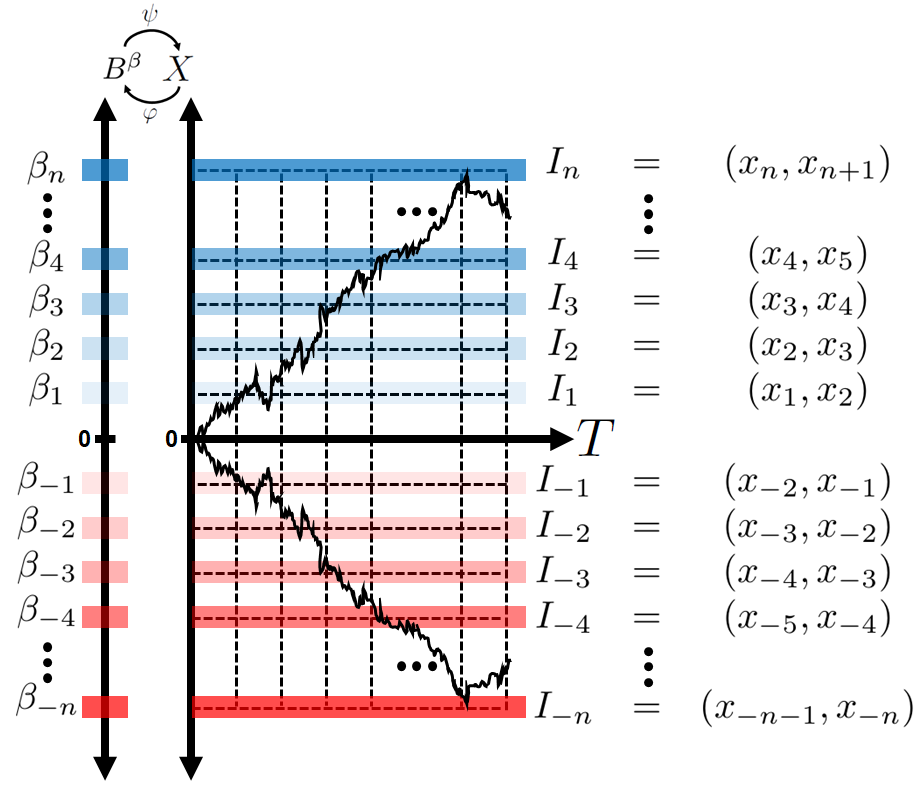}
   \caption{Constructing BGC Stochastic Processes from M-SBM Framework}
   \label{Fig:Framework}
\flushleft
   \textbf{\footnotesize \noindent{As the It\^{o} process reaches further and further intervals $\boldsymbol{I_k}$ from the origin, the intervals' permeability decreases and is scaled by $\boldsymbol{\beta_k}$.
The less permeable the interval becomes, the more it operates as a reflective barrier.
Eventually, there exists an interval that for the given It\^{o} diffusion is effectively fully reflective, forming the hidden barrier of BGC stochastic processes.
}}
\end{figure}

\bigskip \noindent
\begin{theorem}\label{SkewnessParameterofBGCStochasticProcesses}
\textbf{(Skewness Parameter of BGC Stochastic Processes)}. Let us denote by $(X^{(n)}_{t})_t$ the $(\beta_{-n}, ...,\beta_n)$-SBM with drift $\mu$ and barriers $x_{-n}, ..., x_n$, and denote by $(Y_t)_t$ the 2-SBM (i.e. $\beta_1, \beta_2$-SMB) with drift $\mu$, diffusion $\sigma$ and barrier $x_1$.
Let us assume $X^{(n)}_0 = Y_0$, then $X^{(n)}$ converges to $Y$ in the following sense,

\[
 \lim_{n \to \infty} \left\{ \mathbb{E} \Big[ \sup_{s \in [0,t]} |X^{(n)}_s - Y_s| \Big] \right\} = 0, \text{ } \forall t \in [0,T].
\]

\bigskip \noindent
The same holds in the case of $n > 2$ barriers merging.
In this case $(X^{(n)}_t )_t$ is the $(\beta_{-n},...,\beta_n)$-SBM with drift $\mu \in \mathbb{R}$, diffusion $\sigma \in \mathbb{R}$, skewness parameters
$\beta_{-n},...,\beta_n \in [-1, 1]$ and barrier $x_1 \in \mathbb{R}$, $x_{j+1} := \frac{j}{n} + x_1, \text{ } \forall j \in \{ 1,...,n-1 \}$.
Then $\beta=0$. 
\end{theorem}

\begin{proof}
\bigskip \noindent
In contrast to the skewness parameter of the limit $1$-SBM in (\ref{Eq:Beta}), the corresponding skewness parameter of the limit $2$-SBM for BGCSP is given by (\ref{Eq:3_16}),

\begin{equation}\label{Eq:3_16}
\begin{array}{rcl}
\beta  & := & \underbrace{\frac{\displaystyle \prod^{n}_{j=1} (1+\beta_j) - \prod^{n}_{j=1} (1-\beta_j)}
                   {\displaystyle \prod^{n}_{j=1} (1+\beta_j) + \prod^{n}_{j=1} (1-\beta_j)} }_{\boldsymbol{> 0} }
                    +
                   \underbrace{\frac{\displaystyle \prod^{-n}_{j=-1} (1+\beta_j) - \prod^{-n}_{j=-1} (1-\beta_j)}
                   {\displaystyle \prod^{-n}_{j=-1} (1+\beta_j) + \prod^{-n}_{j=-1} (1-\beta_j)} }_{\boldsymbol{< 0} }       \\
        & = & \resizebox{0.9\hsize}{!}{$
\frac{\left[ \displaystyle \prod^{n}_{j=1} (1+\beta_j) - \prod^{n}_{j=1} (1-\beta_j) \right]
\left[ \displaystyle \prod^{-n}_{j=-1} (1+\beta_j) + \prod^{-n}_{j=-1} (1-\beta_j) \right]
+ 
\left[ \displaystyle \prod^{n}_{j=1} (1+\beta_j) + \prod^{n}_{j=1} (1-\beta_j) \right]
\left[ \displaystyle \prod^{-n}_{j=-1} (1+\beta_j) - \prod^{-n}_{j=-1} (1-\beta_j) \right]}
                 {\left[ \displaystyle \prod^{n}_{j=1} (1+\beta_j) + \prod^{n}_{j=1} (1-\beta_j) \right]
\left[ \displaystyle \prod^{-n}_{j=-1} (1+\beta_j) + \prod^{-n}_{j=-1} (1-\beta_j) \right]}
$}                                           \\
        & = & \resizebox{0.9\hsize}{!}{$
\frac{\left[ \displaystyle
 \prod^{n}_{j=1} (1+\beta_j) \prod^{-n}_{j=-1} (1+\beta_j) 
+ \prod^{n}_{j=1} (1+\beta_j) \prod^{-n}_{j=-1} (1-\beta_j) 
- \prod^{n}_{j=1} (1-\beta_j) \prod^{-n}_{j=-1} (1+\beta_j)
- \prod^{n}_{j=1} (1-\beta_j) \prod^{-n}_{j=-1} (1-\beta_j)
 \right] }
{\left[ \displaystyle
 \prod^{n}_{j=1} (1+\beta_j) \prod^{-n}_{j=-1} (1+\beta_j) 
+ \prod^{n}_{j=1} (1+\beta_j) \prod^{-n}_{j=-1} (1-\beta_j) 
+ \prod^{n}_{j=1} (1-\beta_j) \prod^{-n}_{j=-1} (1+\beta_j)
+ \prod^{n}_{j=1} (1-\beta_j) \prod^{-n}_{j=-1} (1-\beta_j)
 \right] }
$}                      \\
&  & +
\resizebox{0.9\hsize}{!}{$
\frac{\left[ \displaystyle
 \prod^{n}_{j=1} (1+\beta_j) \prod^{-n}_{j=-1} (1+\beta_j) 
- \prod^{n}_{j=1} (1+\beta_j) \prod^{-n}_{j=-1} (1-\beta_j) 
+ \prod^{n}_{j=1} (1-\beta_j) \prod^{-n}_{j=-1} (1+\beta_j)
- \prod^{n}_{j=1} (1-\beta_j) \prod^{-n}_{j=-1} (1-\beta_j)
 \right] }
{\left[ \displaystyle
 \prod^{n}_{j=1} (1+\beta_j) \prod^{-n}_{j=-1} (1+\beta_j) 
+ \prod^{n}_{j=1} (1+\beta_j) \prod^{-n}_{j=-1} (1-\beta_j) 
+ \prod^{n}_{j=1} (1-\beta_j) \prod^{-n}_{j=-1} (1+\beta_j)
+ \prod^{n}_{j=1} (1-\beta_j) \prod^{-n}_{j=-1} (1-\beta_j)
 \right] }.
$}
\end{array}
\end{equation}

\bigskip \noindent
Noting that due to the symmetry of BGCSP about the origin,

\[
 \displaystyle   \prod^{-n}_{j=-1} (1-\beta_j) = \prod^{n}_{j=1} (1+\beta_j), \quad
 \displaystyle   \prod^{-n}_{j=-1} (1+\beta_j) = \prod^{n}_{j=1} (1-\beta_j),
\]

\bigskip \noindent
which allows the $\prod^{n}_{j=1} (1+\beta_j) $ terms to factor out in (\ref{Eq:3_16}) giving,

\begin{equation*}
\begin{array}{rcl}
\beta 
        & = & \resizebox{0.9\hsize}{!}{$
\frac{ \displaystyle   \prod^{n}_{j=1} (1+\beta_j) \left[ 
 \prod^{-n}_{j=-1} (1+\beta_j) 
+  \prod^{-n}_{j=-1} (1-\beta_j) 
- \prod^{n}_{j=1} (1-\beta_j) 
   - \left. \prod^{n}_{j=1} (1-\beta_j) \prod^{-n}_{j=-1} (1+\beta_j)  \middle/ \prod^{n}_{j=1} (1+\beta_j) \right. \right] }
{\displaystyle \prod^{n}_{j=1} (1+\beta_j) \left[ 
  \prod^{-n}_{j=-1} (1+\beta_j) 
+  \prod^{-n}_{j=-1} (1-\beta_j) 
+ \prod^{n}_{j=1} (1-\beta_j) 
  + \left. \prod^{n}_{j=1} (1-\beta_j) \prod^{-n}_{j=-1} (1+\beta_j)  \middle/ \prod^{n}_{j=1} (1+\beta_j) \right. \right] }
$}                      \\
        &   & + \resizebox{0.9\hsize}{!}{$
\frac{ \displaystyle \prod^{n}_{j=1} (1+\beta_j)  \left[ 
  \prod^{-n}_{j=-1} (1+\beta_j) 
-  \prod^{-n}_{j=-1} (1-\beta_j) 
- \prod^{n}_{j=1} (1-\beta_j) 
  + \left. \prod^{n}_{j=1} (1-\beta_j) \prod^{-n}_{j=-1} (1+\beta_j)  \middle/ \prod^{n}_{j=1} (1+\beta_j) \right.  \right] }
{\displaystyle \prod^{n}_{j=1} (1+\beta_j) \left[ 
  \prod^{-n}_{j=-1} (1+\beta_j) 
+  \prod^{-n}_{j=-1} (1-\beta_j) 
+ \prod^{n}_{j=1} (1-\beta_j) 
  + \left. \prod^{n}_{j=1} (1-\beta_j) \prod^{-n}_{j=-1} (1+\beta_j)  \middle/ \prod^{n}_{j=1} (1+\beta_j) \right. \right] }
$}
\end{array}
\end{equation*}

\bigskip \noindent
which expands to,

\begin{equation}
\begin{array}{rcl}
\beta
        & = & \resizebox{0.9\hsize}{!}{$
\frac{ \displaystyle
    \left[ 
 \prod^{-n}_{j=-1} (1-\beta_j) 
   - \left. \left( \prod^{n}_{j=1} (1-\beta_j) \right)^2   \middle/ \prod^{n}_{j=1} (1+\beta_j) \right. \right] }
{\displaystyle
  \left[ 
  \prod^{-n}_{j=-1} (1+\beta_j) 
+  \prod^{-n}_{j=-1} (1-\beta_j) 
+ \prod^{n}_{j=1} (1-\beta_j) 
  + \left. \left( \prod^{n}_{j=1} (1-\beta_j)  \right)^2  \middle/ \prod^{n}_{j=1} (1+\beta_j) \right. \right] }
$}                      \\
        &   & + \resizebox{0.9\hsize}{!}{$
\frac{ \displaystyle
   \left[ 
- \prod^{-n}_{j=-1} (1-\beta_j) 
  + \left. \left( \prod^{n}_{j=1} (1-\beta_j) \right)^2  \middle/ \prod^{n}_{j=1} (1+\beta_j) \right.  \right] }
{\displaystyle
  \left[ 
  \prod^{-n}_{j=-1} (1+\beta_j) 
+  \prod^{-n}_{j=-1} (1-\beta_j) 
+ \prod^{n}_{j=1} (1-\beta_j) 
  + \left. \left( \prod^{n}_{j=1} (1-\beta_j) \right)^2  \middle/ \prod^{n}_{j=1} (1+\beta_j) \right. \right] }.
$}
\end{array}
\end{equation}

\bigskip \noindent
It is clear that the numerator equates to 0 and so $\beta = 0$, completing the proof.
\end{proof}

\bigskip
\begin{remark}
Due to the bi-directionality of BGC stochastic processes, then $n$ in (\ref{Eq:Even}) is always even, so $\beta =0$.
From \citet{Portenko1976}, if $|x_i| \leq 1$, then $| \sum^{n}_{i=1} \alpha_i | \geq 1$ is of special interest.
With BGCSP, $\alpha_{-n} + \alpha_{n} = 0$, $\alpha_{-n+1} + \alpha_{n-1} = 0$,...,$\alpha_{-1} + \alpha_{1} = 0$ due to their symmetry about the origin, hence $| \sum^{n}_{i=1} \alpha_i | = 0$ as well.
\end{remark}

\bigskip
\begin{theorem}
\textbf{(Cylindrical BGCSPs are 2-SBMs)}.
For a complete filtered probability space $(\Omega, \mathcal{F}, \{ \mathcal{F} \}_{t \geq 0}, \mathbb{P})$ and a BGC function $\Psi (y) : \mathbb{R} \rightarrow \mathbb{R}$, $\forall y \in \mathbb{R}$, then the corresponding BGC It\^{o} diffusion is defined as follows,

\bigskip
\begin{equation}
      dY_t  =  f(Y_t, t) \, dt + g(Y_t, t)  \, dW_t  -  \underbrace{ \sgn[Y_t] \Psi  (Y_t, t) }_{\textbf{BGC}},
\end{equation}

\bigskip \noindent
where $f(Y_t, t)$ is the drift coefficient, $g(Y_t, t) $ is the diffusion coefficient, $\sgn[x]$ is the usual sign function, $f(Y_t, t)$,  $g(Y_t, t)$, $\Psi (Y_t, t)$ are convex functions and the 2-SBM is defined by,

\begin{eqnarray}\label{Eqn:3_9_12}
 & & {
\begin{cases}
   \begin{array}{rcl}
\displaystyle  dX_t           & = & \mu \, dt + \sigma \, d W_t 
 +  \underbrace{\beta_{-1} \, dL^{x_{-1} }_{t} }_{<0}
 +  \underbrace{\beta_{1} \, dL^{x_{1} }_{t} }_{>0}\\
\displaystyle  X_0            & = & 0,  \text{ } \mathcal{E} \Big( \mu, \sigma, ( \beta_{-1},  \beta_{1} ) \Big) \\
\displaystyle L^{x_j}_{t} & = & \int ^{t}_{0} \mathds{1}_{ \{ X_s = z_j \} } \, d L^{x_j}_{s},  \text{ }    j \in \{ -1,1 \}
   \end{array},
\end{cases}
}
\end{eqnarray}

\bigskip \noindent
then, $Y_t \rightarrow X_t$ almost surely.
\end{theorem}
\begin{proof}
It is conceivable that under general non-constant $f(X_t, t)$ and $g(X_t, t)$ and some generalized BGC function $\Psi '(f(X_t, t), g(X_t, t), X_t, t)$ that $\Psi '(x)$ could modulate $X_t$ such that it is bounded above and below by a constant barrier at $a$ and $b$, respectively, where $b=-a$. 
For this theorem, we are required to prove that constant over time (i.e. cylindrical) BGC functions $\Psi (X_t, t)$ will converge almost surely to a 2-SBM. 
We know from at least \citet{Krykun2017} that if $|\alpha_i| \leq 1$, $i \in \{ 1,...,n \}$, there exists a strong solution to (\ref{Eqn:3_9_12}).
Since the BGC functions $\Psi (X_t,t) \in \mathbb{R}$ are convex, there exists some value $\kappa$ for both a hidden lower barrier $\mathfrak{B}_L$ and a hidden upper barrier $\mathfrak{B}_L$ that are induced by $\Psi (X_t,t)$.
For BGCSP, there is no fully reflective barrier defined in advance as there is with M-SBM.
However, there are still two fully reflective barriers in BGCSP because the BGC term $\Psi(X_t, t)$ will enable the constrained It\^{o} process $Y_t$ to eventually be overtaken by the underlying unconstrained It\^{o} process $X_t$ such that $|X^{(n)}_{s}| \geq |Y_{s}|$ giving,

\begin{equation}
  \mathfrak{B}_U = \kappa \text{ for } \lim_{n \uparrow \kappa} \left\{ \mathbb{E} \Big( \sup_{s \in [0,t]} \Big| |X^{(n)}_s| - |Y_s| \Big| \Big) \right\} = 0,
\end{equation}
\begin{equation}
  \mathfrak{B}_L = -\kappa \text{ for } \lim_{n \downarrow -\kappa} \left\{ \mathbb{E} \Big( \sup_{s \in [0,t]} \Big| |X^{(n)}_s| - |Y_s| \Big| \Big) \right\} = 0.
\end{equation}

\bigskip \noindent
For this to be true, it must be shown that $\kappa > 0$ exists.
We create a small neighborhood $\mathcal{N}$ about the initial point $x_0$ of radius $\epsilon \in \mathbb{R}_{+}$ such that $\mathcal{N}(x_0) = (x_0 - \epsilon, x_0 + \epsilon)$.
As $\epsilon \rightarrow +\infty$, $\mathfrak{B}_L$ and $\mathfrak{B}_U$ will eventually lie in $\mathcal{N}(x_0)$.

\bigskip \noindent
If $x_0 > 0$, then $\sup\{ \mathcal{N}(x_0) \} = x_0 + \min(\epsilon)$ such that $\mathfrak{B}_U = \sup\{ \mathcal{N}(x_0) \} = \kappa$.\\
If $x_0 = 0$, then $\sup\{ \mathcal{N}(x_0) \} = \inf\{ \mathcal{N}(x_0) \}$ such that $\mathfrak{B}_L = -\mathfrak{B}_U = | \mathfrak{B}_U | = \kappa$.\\
If $x_0 < 0$, then $\inf\{ \mathcal{N}(x_0) \} = x_0 - \min(\epsilon)$ such that $\mathfrak{B}_L = \inf\{ \mathcal{N}(x_0) \} = -\kappa$.

\bigskip \noindent
Hence $\kappa$ exists and its value is $\kappa = f(\Psi(X_t, t), X_t, t, \mu, \sigma)$ for some function $f: \mathbb{R} \rightarrow \mathbb{R}$. 
Having found $\kappa$, we know that the reflectiveness at $\pm \kappa$, i.e. $| \beta_{\kappa} | = 1$, $| \beta_{- \kappa} | = 1$ and before $\pm \kappa$, i.e. $| \beta_{i} | < 1$, $| \beta_{- i} | < 1$.
Hence, $\beta_{-\kappa}, ...,\beta_{-1}, \beta_{1},..., \beta_{\kappa}$ for $X_t$ must be scaled for $Y_t$ by $\Psi(X_t, t)$ and since $\Psi(X_t, t)$ is strictly convex and symmetrical about the origin, then the ordering is preserved,

\begin{equation}\label{Eq:Sequence}
 \frac{\beta_{-\kappa} }{\Psi(\kappa, t)} < ... < \frac{\beta_{-1} }{\Psi(\kappa, t)} < \frac{\beta_{1} }{\Psi(\kappa, t)} <...< \frac{\beta_{\kappa} }{\Psi(\kappa, t)}.
\end{equation}

\bigskip \noindent
(\ref{Eq:Sequence}) ensures that a strong solution to BGCSP exists within a 2-SBM framework, completing the proof.
\end{proof}

\bigskip \noindent
So far, our formulations of BGC functions have been expressed in the general form $\Psi(X_t, t)$, but we have considered BGC barriers induced by time-independent convex surfaces which can be specified by just $\Psi(X_t)$, hence M-SBM is related to BGCSPs with $\Psi(X_t)$.
However, since the barriers have been specified to be able to change not only under space (distance) but over time as well, we demonstrate this additional complexity of BGCSPs in Figure \ref{Fig:Framework2}.

\begin{figure}[ht]
   \centering
 \includegraphics[width=\linewidth]{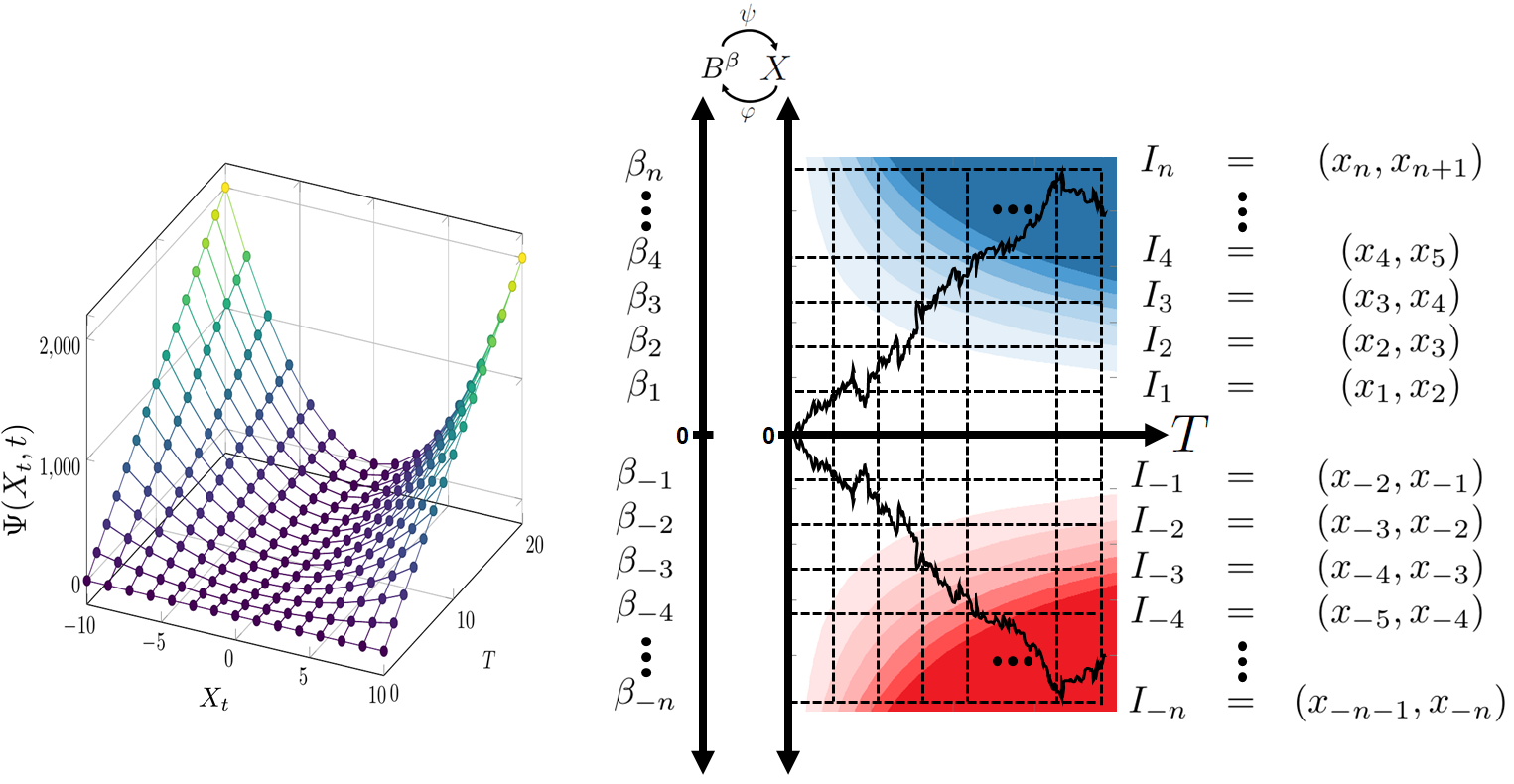}
   \caption{Example BGC Function $\Psi(X_t, t)$ Constraining BGCSPs over Space and Time, more so than in M-SBM}
   \label{Fig:Framework2}
\flushleft
   \textbf{\footnotesize \noindent{BGCSPs are more expansive than M-SBMs (compare with Figure \ref{Fig:StandardFramework}) in the sense that the barriers can change over time, hence $\boldsymbol{\Psi(X_t, t)}$ rather than the simpler $\boldsymbol{\Psi(X_t)}$, which is related to M-SBM.
The mesh plot on the left induces or dictates how the BGCSP on the right is constrained, showing that the M-SBM doesn't cover such time-dependent constraints.
}}
\end{figure}
\FloatBarrier

\bigskip \noindent
Having developed the M-SBM and 2-SBM frameworks for BGCSP, we can now support this by numerical simulations in the Results and Discussion section.

\bigskip
\section{Results and Discussion}


\noindent
In the following simulations, the underlying unconstrained It\^{o} diffusions have drift $\mu = 0$ and diffusion $\sigma = 1$, resulting in just the Wiener process.
This is so that the subsequent impact of BGC can be easily compared.
Despite this, we still refer to these as the more general It\^{o} diffusions because these parameters can be modified for one's specific requirements.

\bigskip \noindent
To validate the existing M-SBM research and to support our comparison of BGCSP with M-SBM, we develop Algorithm \ref{Alg:Algo2} which is used to progressively introduce additional reflective barriers.
In the subsequent series of simulations, we introduce 2, 4, 8, 16 and finally 32 semipermeable barriers, with increasing reflectiveness (i.e. decreasing permeability) the further the It\^{o} diffusion is from the origin, which are simulated via Algorithm \ref{Alg:Algo2}.

\begin{algorithm}\label{Alg:Algo2}
\scriptsize
\caption{Approximating BGC Stochastic Processes via Successive Reflective Barriers}
\# Pseudocode based on R\\
\textbf{INPUT: }\\
$\mu=drift,\text{ }\sigma=diffusion,\text{ }i=simulation\text{ }index,\text{ }$s$=\#\text{ } simulations=10,000,\text{ }t=time\text{ }steps=1001,\text{ }j=time\text{ }index,\text{ }Print\_Simulations=TRUE$\\
\textbf{OUTPUT: }\\
$ID\_value \leftarrow  matrix(0:0, nrow = TimeSteps,   ncol = Simulations)$\\
$T\_1000   \leftarrow  matrix(0:0, nrow = Simulations, ncol = 1)$\\
\For{(i=1:s)}{%
   \For{(j=1:t)}{%
      \uIf{(t$==$1)}{%
         $ID\_value[t,i] \leftarrow 0$
      } \Else {
         $dt = (t/TimeSteps)$\\
         $ID\_value[t,i] \leftarrow ( \mu * dt + \sigma * rnorm(1) )$\\
         $Sum\_ID\_value \leftarrow sum(ID\_value[,i])$\\
         $\text{ }$\\
         $\#$  UPPER BARRIERS$================================$\\
         \uIf{$( (Sum\_ID\_value > 0) \text{ }\&\& \text{ }(Sum\_ID\_value <= UpperBarrier\_01) )$} {
            $\text{Do nothing}$\;
         } \uElseIf {$( (Sum\_ID\_value > UpperBarrier\_01) \text{ }\&\&\text{ } (Sum\_ID\_value <= UpperBarrier\_02) )$} {
           $ID\_value[t,i] \leftarrow ( ID\_value[t,i] -  abs(ID\_value[t,i] * ID\_value[t,i])/100 )$
         } \uElseIf {$( Sum\_ID\_value > UpperBarrier\_16 )$} {
            $ID\_value[t,i] \leftarrow ( - abs(ID\_value[t,i]) )$
         } \Else {
            $ID\_value[t,i] \leftarrow ID\_value[t,i]$
         }
         $\text{ }$\\
         $\#$  LOWER BARRIERS$================================$\\
         \uIf{$( (Sum\_ID\_value < 0) \text{ }\&\&\text{ } (Sum\_ID\_value >= LowerBarrier\_01) )$} {
            $\text{Do nothing}$\;
         } \uElseIf {$( Sum\_ID\_value < LowerBarrier\_16 )$} {
            $ID\_value[t,i] \leftarrow ( abs(ID\_value[t,i]) )$
         } \Else {
            $ID\_value[t,i] \leftarrow ID\_value[t,i]$
         }
      }
      \If{(Print\_Simulations$==$TRUE)} {%
         \uIf{(i$==$1)} {%
            $plot(T, cumsum(ID_value[,i]), type = "l", ylim=c(yMax, yMin) )$
         } \Else {
            $lines(T, cumsum(ID_value[,i]), type = "l", ylim=c(yMax, yMin) )$
         }
      }
   }
   $T\_1000[i] <- sum(ID\_value[,i])$
}
\end{algorithm}


\bigskip \noindent
The simplest application of Algorithm \ref{Alg:Algo2} is shown for two fully reflective barriers in Figure \ref{Fig:10000Simulationsof1000Step1DimensionalItoDiffusionsWith2ReflectiveBarriers}.

\begin{figure}[H]
   \centering
 \includegraphics[width=\linewidth/2 -\linewidth/50]{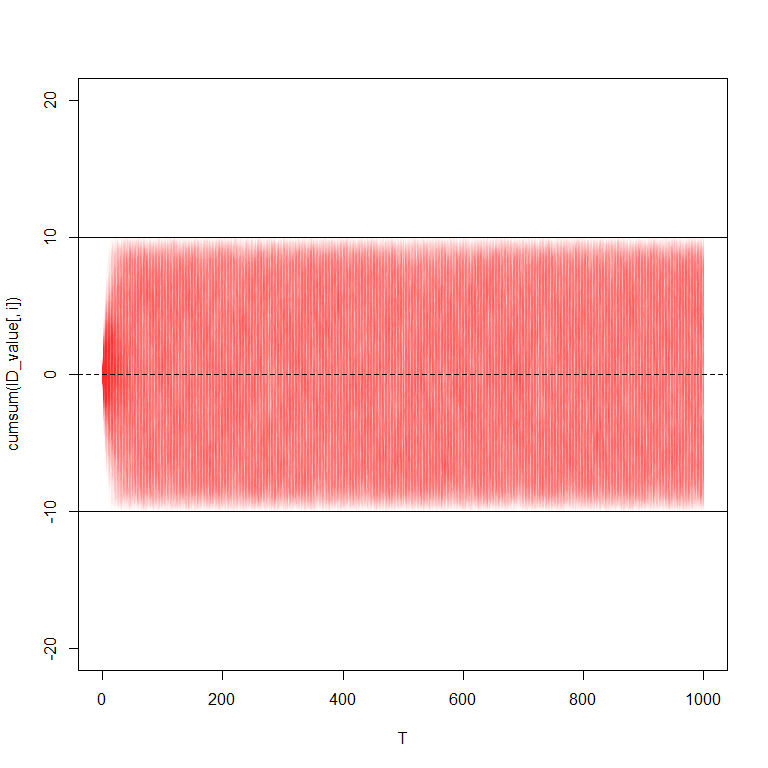}
 \includegraphics[width=\linewidth/2 -\linewidth/50]{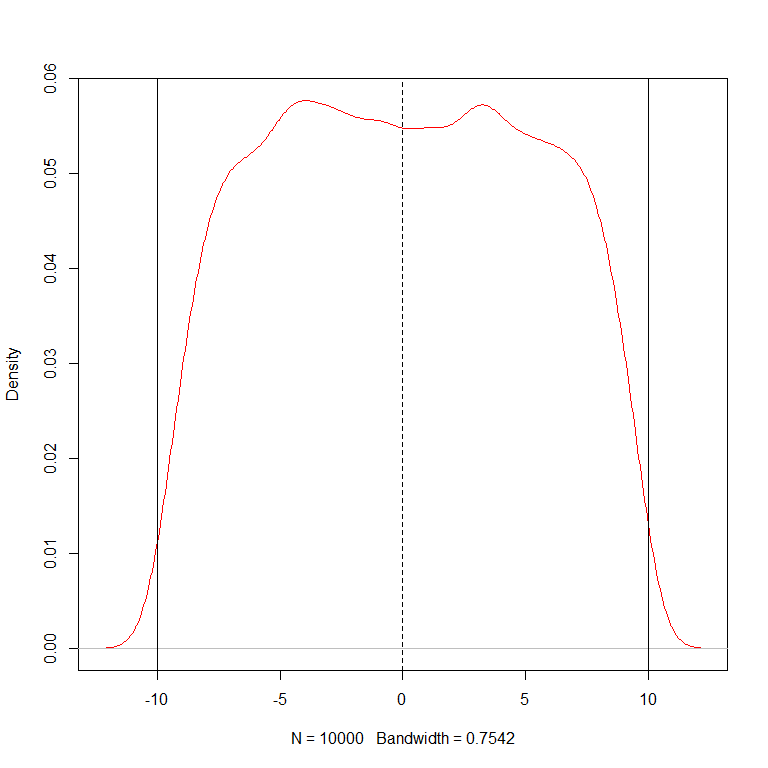}\\ 
  \textbf{\footnotesize \noindent{(a). 1,000 Simulations, \quad  (b). 10,000 Simulation Density}}
   \caption{10,000 Simulations of 1,000 Step 1-Dimensional It\^{o} Diffusions With 2 Reflective Barriers}
   \label{Fig:10000Simulationsof1000Step1DimensionalItoDiffusionsWith2ReflectiveBarriers}
\end{figure}
\FloatBarrier

\bigskip \noindent
Figure \ref{Fig:10000Simulationsof1000Step1DimensionalItoDiffusionsWith2ReflectiveBarriers} has 2 fully reflective barriers at $\pm 10$ generated using Algorithm 1.
This was then increased to 4 barriers (2 fully reflective and 2 semipermeable) as shown in Figure \ref{Fig:10000Simulationsof1000Step1DimensionalItoDiffusionsWith4ReflectiveBarriers}.

\begin{figure}[H]
   \centering
 \includegraphics[width=\linewidth/2 -\linewidth/50]{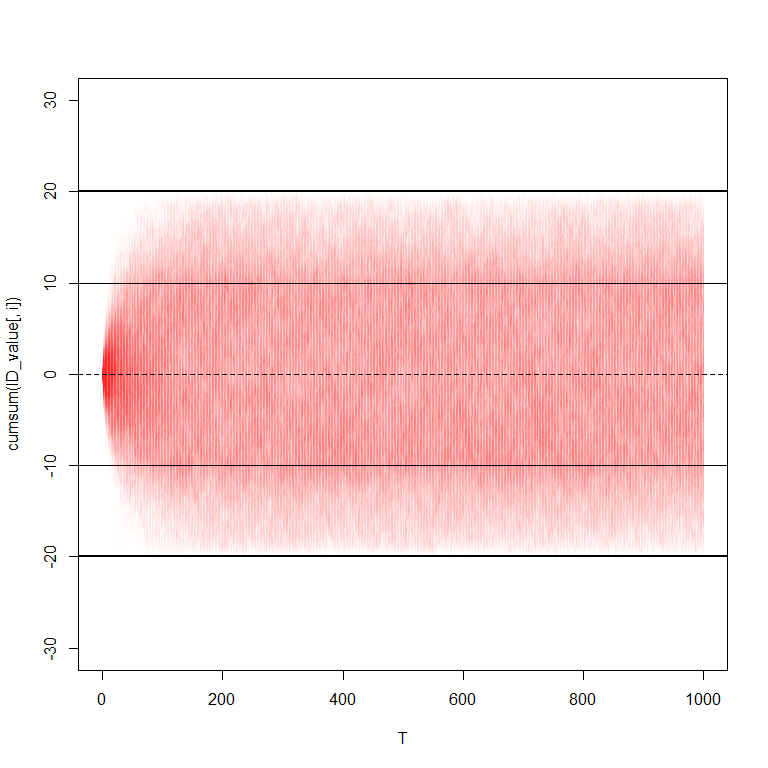}
 \includegraphics[width=\linewidth/2 -\linewidth/50]{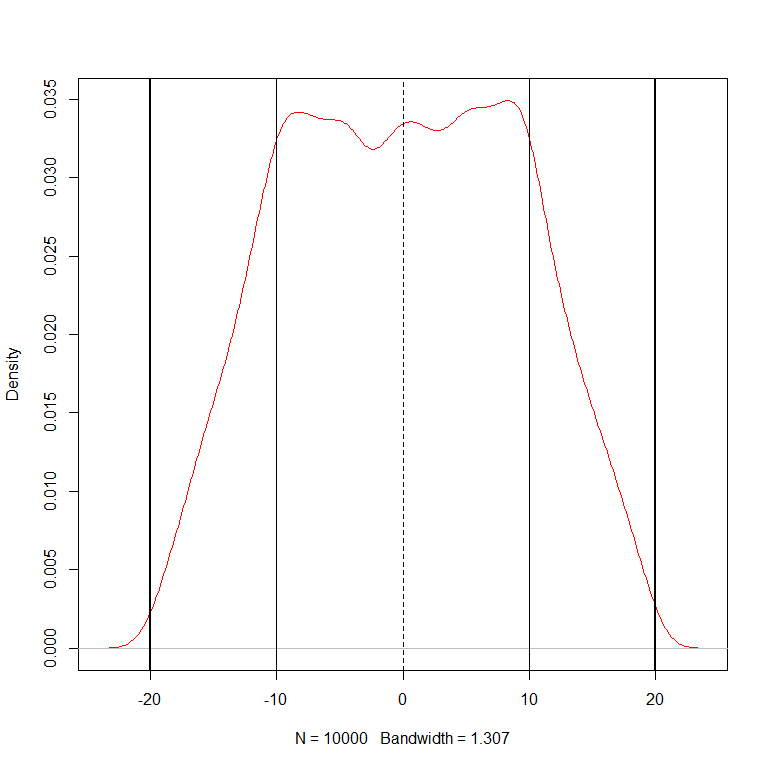}\\ 
  \textbf{\footnotesize \noindent{(a). 1,000 Simulations, \quad  (b). 10,000 Simulation Density}}
   \caption{10,000 Simulations of 1,000 Step 1-Dimensional It\^{o} Diffusions With 4 Reflective Barriers}
   \label{Fig:10000Simulationsof1000Step1DimensionalItoDiffusionsWith4ReflectiveBarriers}
\end{figure}
\FloatBarrier

\bigskip \noindent
In Figure \ref{Fig:10000Simulationsof1000Step1DimensionalItoDiffusionsWith4ReflectiveBarriers}, we make the barriers at $\pm 20$ fully reflective and the barriers at $\pm 10$ now to be semipermeable.
Although it may not yet be apparent due to the thickness of the drawn barriers, we have and will continue to increase the thickness of the barriers to highlight the increasing reflectiveness (and decreasing semipermeability) the further the It\^{o} diffusions are from the origin.
The total number of barriers was doubled again to result in 8 barriers (2 fully reflective and 6 semipermeable), as shown in Figure \ref{Fig:10000Simulationsof1000Step1DimensionalItoDiffusionsWith8ReflectiveBarriers}.

\begin{figure}[H]
   \centering
 \includegraphics[width=\linewidth/2 -\linewidth/50]{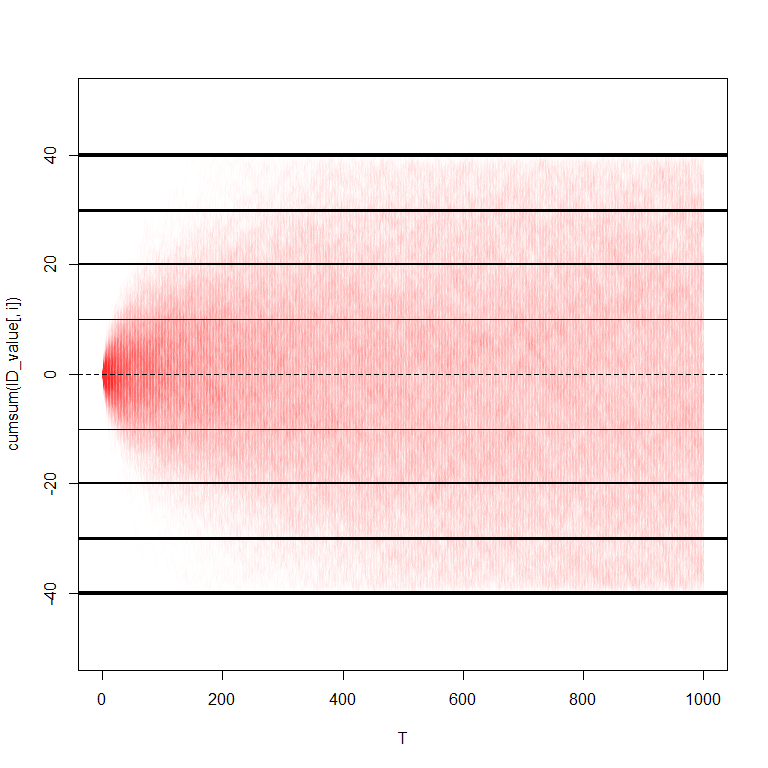}
 \includegraphics[width=\linewidth/2 -\linewidth/50]{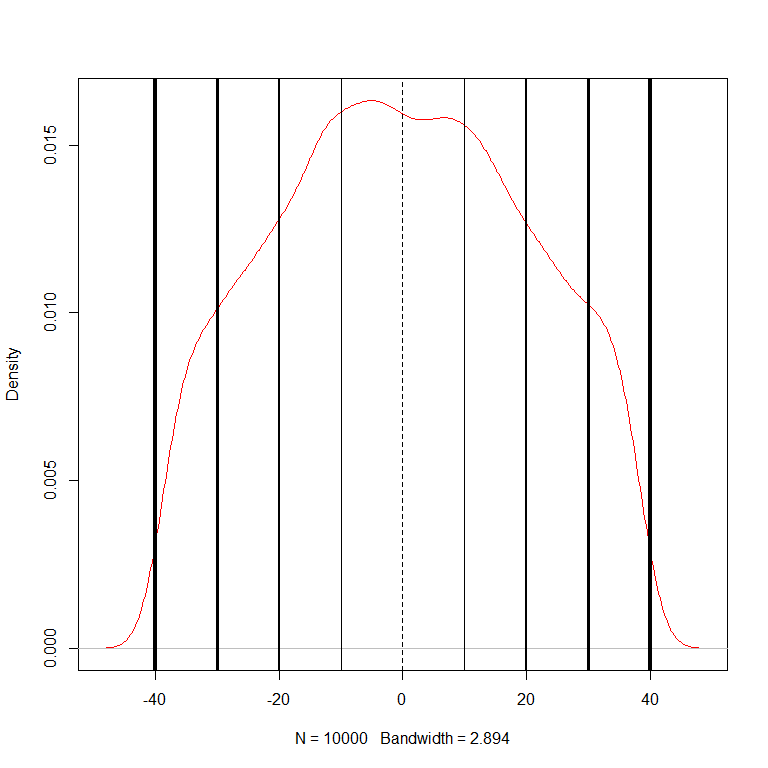}\\ 
  \textbf{\footnotesize \noindent{(a). 1,000 Simulations, \quad  (b). 10,000 Simulation Density}}
   \caption{10,000 Simulations of 1,000 Step 1-Dimensional It\^{o} Diffusions With 8 Reflective Barriers}
   \label{Fig:10000Simulationsof1000Step1DimensionalItoDiffusionsWith8ReflectiveBarriers}
\end{figure}
\FloatBarrier

\bigskip \noindent
In Figure \ref{Fig:10000Simulationsof1000Step1DimensionalItoDiffusionsWith8ReflectiveBarriers}(b), we start to notice a corrugation or `crinkling' of the density due to 2 fully reflective barriers and 6 semipermeable barriers.
This was doubled again to 16 barriers (2 fully reflective and 14 semipermeable), as shown in Figure \ref{Fig:10000Simulationsof1000Step1DimensionalItoDiffusionsWith16ReflectiveBarriers}.

\begin{figure}[H]
   \centering
 \includegraphics[width=\linewidth/2 -\linewidth/50]{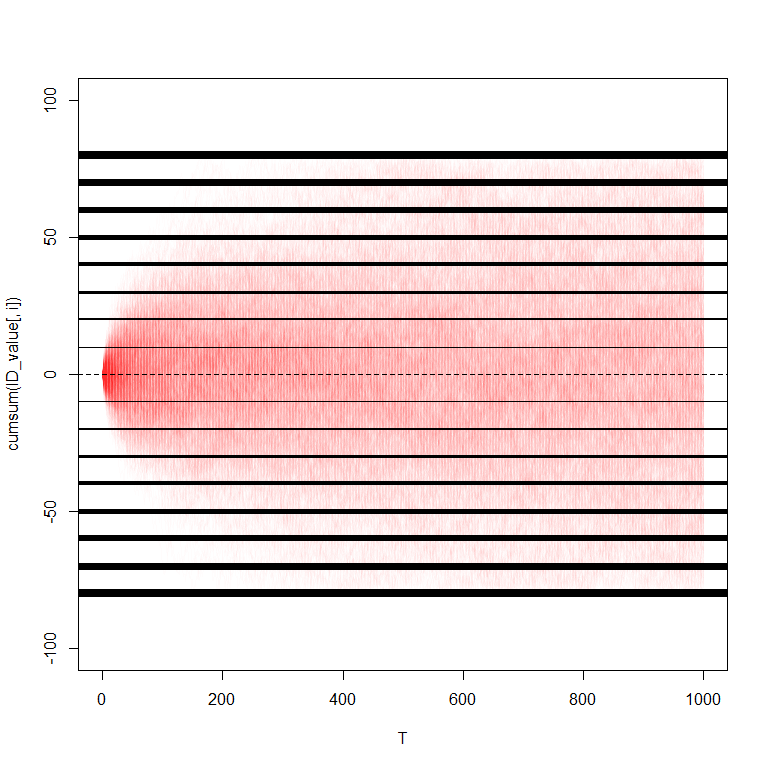}
 \includegraphics[width=\linewidth/2 -\linewidth/50]{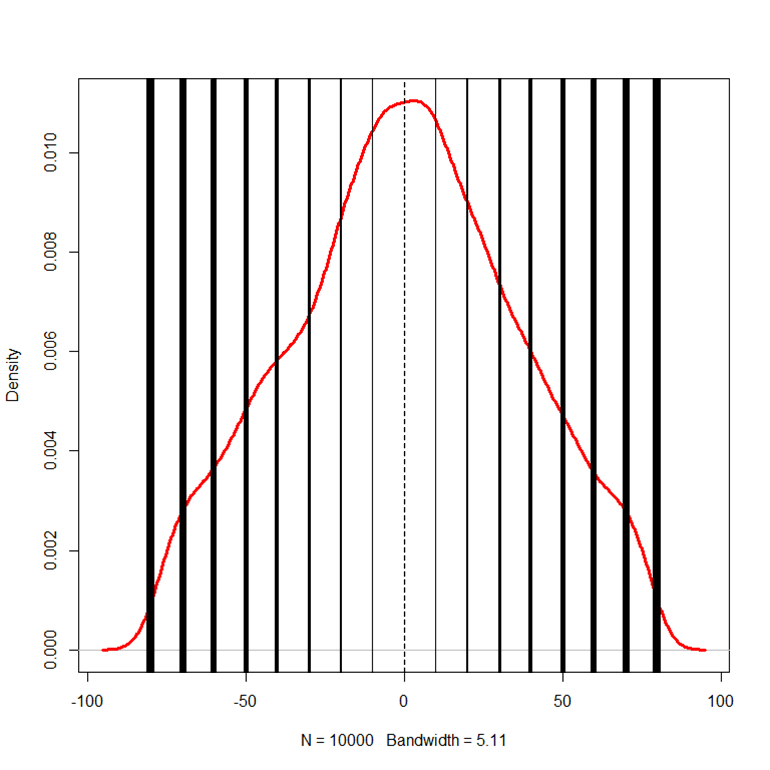}\\ 
  \textbf{\footnotesize \noindent{(a). 1,000 Simulations, \quad  (b). 10,000 Simulation Density}}
   \caption{10,000 Simulations of 1,000 Step 1-Dimensional It\^{o} Diffusions With 16 Reflective Barriers}
   \label{Fig:10000Simulationsof1000Step1DimensionalItoDiffusionsWith16ReflectiveBarriers}
\end{figure}
\FloatBarrier

\bigskip \noindent
In Figure \ref{Fig:10000Simulationsof1000Step1DimensionalItoDiffusionsWith16ReflectiveBarriers}(b), we notice that the corrugation effect has become more pronounced due to another doubling of the number of barriers.
This was doubled again to 32 barriers (2 fully reflective and 30 semipermeable), as shown in Figure \ref{Fig:10000Simulationsof1000Step1DimensionalItoDiffusionsWith32ReflectiveBarriers}.

\begin{figure}[H]
   \centering
 \includegraphics[width=\linewidth/2 -\linewidth/50]{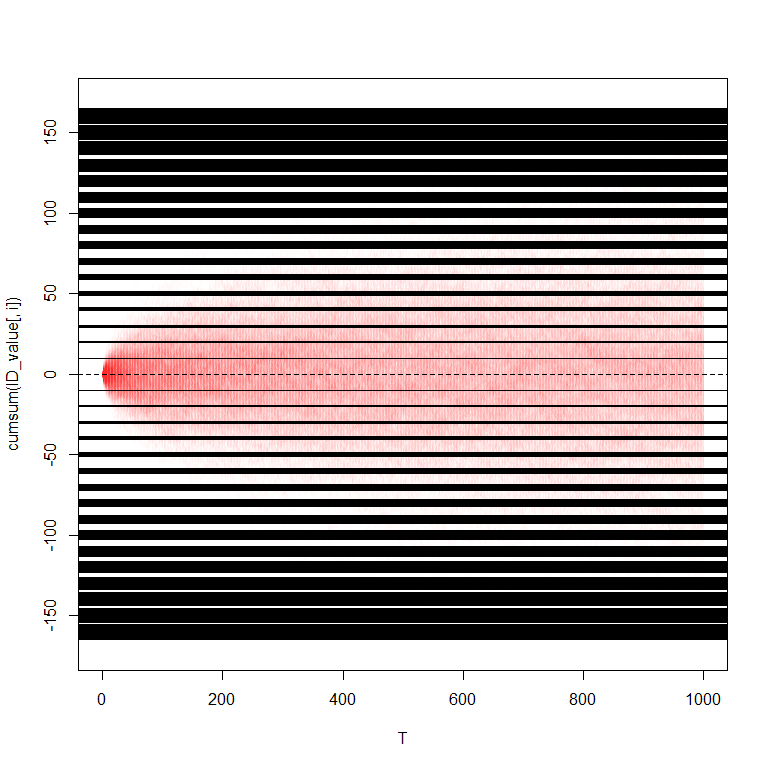}
 \includegraphics[width=\linewidth/2 -\linewidth/50]{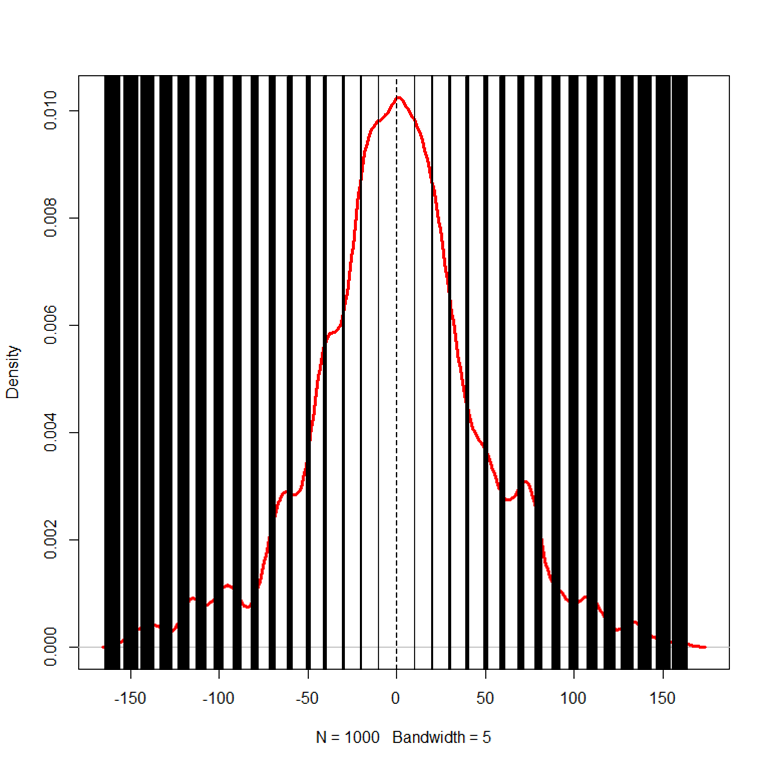}\\ 
  \textbf{\footnotesize \noindent{(a). 1,000 Simulations, \quad  (b). 10,000 Simulation Density}}
   \caption{10,000 Simulations of 1,000 Step 1-Dimensional It\^{o} Diffusions With 32 Reflective Barriers}
   \label{Fig:10000Simulationsof1000Step1DimensionalItoDiffusionsWith32ReflectiveBarriers}
\end{figure}
\FloatBarrier

\bigskip \noindent
Finally, in Figure \ref{Fig:10000Simulationsof1000Step1DimensionalItoDiffusionsWith32ReflectiveBarriers}(b), we notice the most amount of the corrugation effect due to yet another doubling of the number of semipermeable barriers.
Due to the importance of this density as a sufficient approximation of BGC densities, we plot the density again in Figure \ref{Fig:10000Simulationsof1000Step1DimensionalItoDiffusionsWith32ReflectiveBarriersClean} without the barriers depicted and slightly larger.

\begin{figure}[H]
   \centering
 \includegraphics[width=\linewidth]{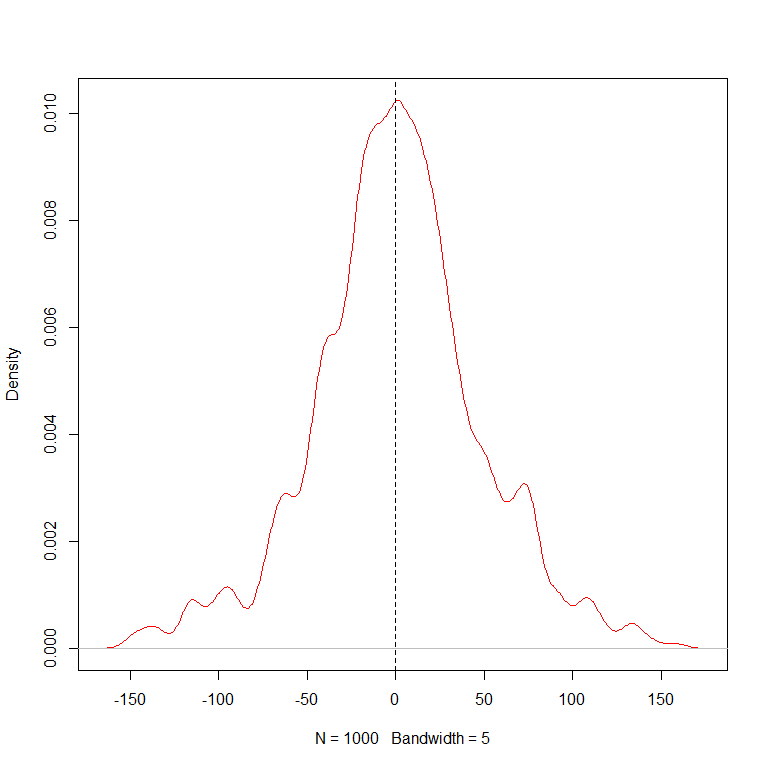}
   \caption{Typical Density of BGC It\^{o} Diffusions Approximated by 10,000 Simulations of 1,000 Step 1-Dimensional It\^{o} Diffusions With 32 Reflective Barriers}
   \label{Fig:10000Simulationsof1000Step1DimensionalItoDiffusionsWith32ReflectiveBarriersClean}
\end{figure}

\bigskip \noindent
Figure \ref{Fig:10000Simulationsof1000Step1DimensionalItoDiffusionsWith32ReflectiveBarriersClean} shows that after 32 barriers, we effectively arrive at the typical density of BGC It\^{o} diffusions, which have an infinite number of increasingly reflective barriers, with just 2 fully reflective hidden barriers.
If we take the limit of this numerical approximation process, the number of such barriers $n$ would approach infinitely many barriers, of smaller and smaller size and hence smaller constraining contribution.
These approximation barriers are thus replaced by the main BGC function itself, $\Psi (X_t, t) = (\frac{X_t}{10})^2$, as shown in Figure \ref{Fig:10000Simulationsof1000StepBGC1DimensionalItoDiffusions}, where the algorithm for BGCSP was stated in \cite{TarantoKhan2020_2}.

\begin{figure}[H]
   \centering
 \includegraphics[width=\linewidth/2 -\linewidth/50]{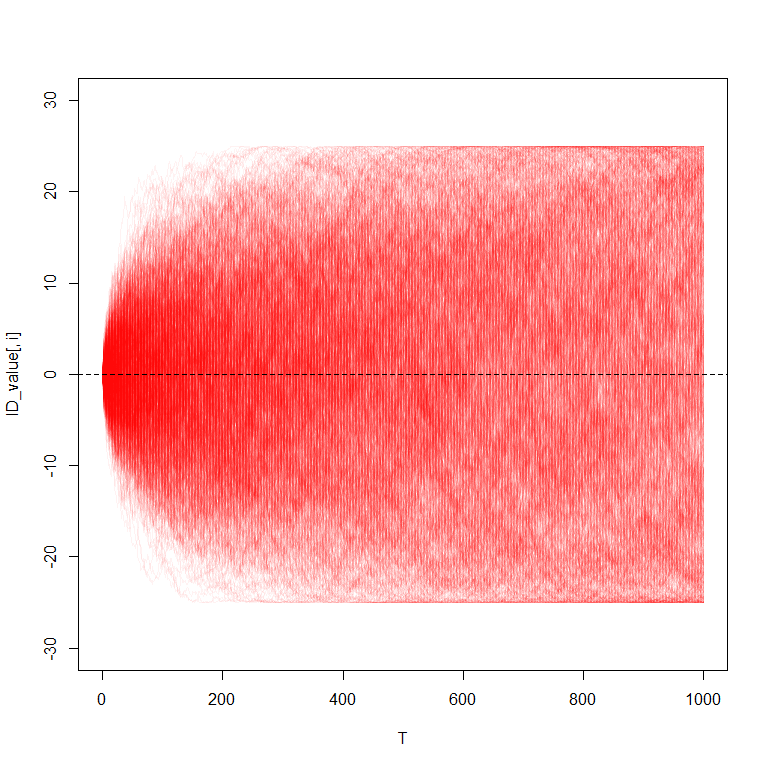}
 \includegraphics[width=\linewidth/2 -\linewidth/50]{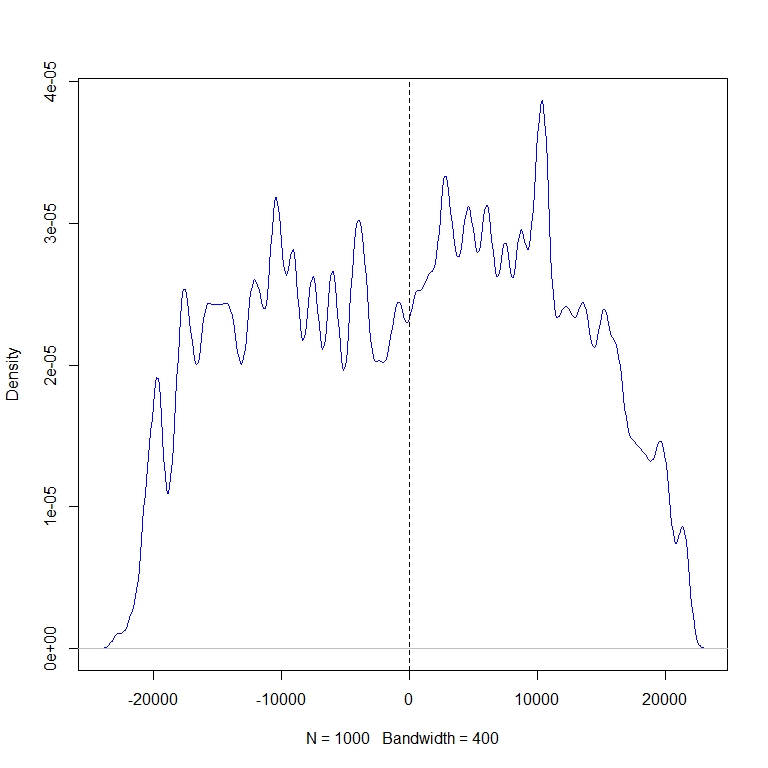}\\ 
  \textbf{\footnotesize \noindent{\quad (a). 1,000 Simulations, \quad \quad \quad (b). 10,000 Simulation Density}}
   \caption{10,000 Simulations of 1,000 Step BGC 1-Dimensional It\^{o} Diffusions}
   \label{Fig:10000Simulationsof1000StepBGC1DimensionalItoDiffusions}
\end{figure}

\bigskip \noindent
From Figure \ref{Fig:10000Simulationsof1000StepBGC1DimensionalItoDiffusions}, we see the typical characteristics of BGCSPs; 1). a certain amount of discretization or banding at various local times, 2). the emergence of two hidden reflective barriers that are not known exactly in advance and can only be estimated, 3). the density is `corrugated' or `rough'.

\bigskip \noindent
The random component of the It\^{o} diffusions, i.e. the $dW_t$ term is sampled from a standard normal distribution that is then constrained by BGCSP.
The density of Figure \ref{Fig:10000Simulationsof1000StepBGC1DimensionalItoDiffusions}(b) has no discontinuities.
However, if we sample the path increments from a discrete binary (i.e. binomial) random distribution, we obtain a random walk that is then constrained by BGCSP.
When a histogram is derived for the corresponding simulated data rather than fitting a density through the distribution, we obtain Figure \ref{Fig:10000Simulationsof1000StepUnconstrained2DimensionalItoDiffusionsWithoutColoration}, which shows the discontinuities, also evident in \cite{TarantoKhan2020_1}.

\begin{figure}[ht]
   \centering
 \includegraphics[width=\linewidth/2 - \linewidth/70]{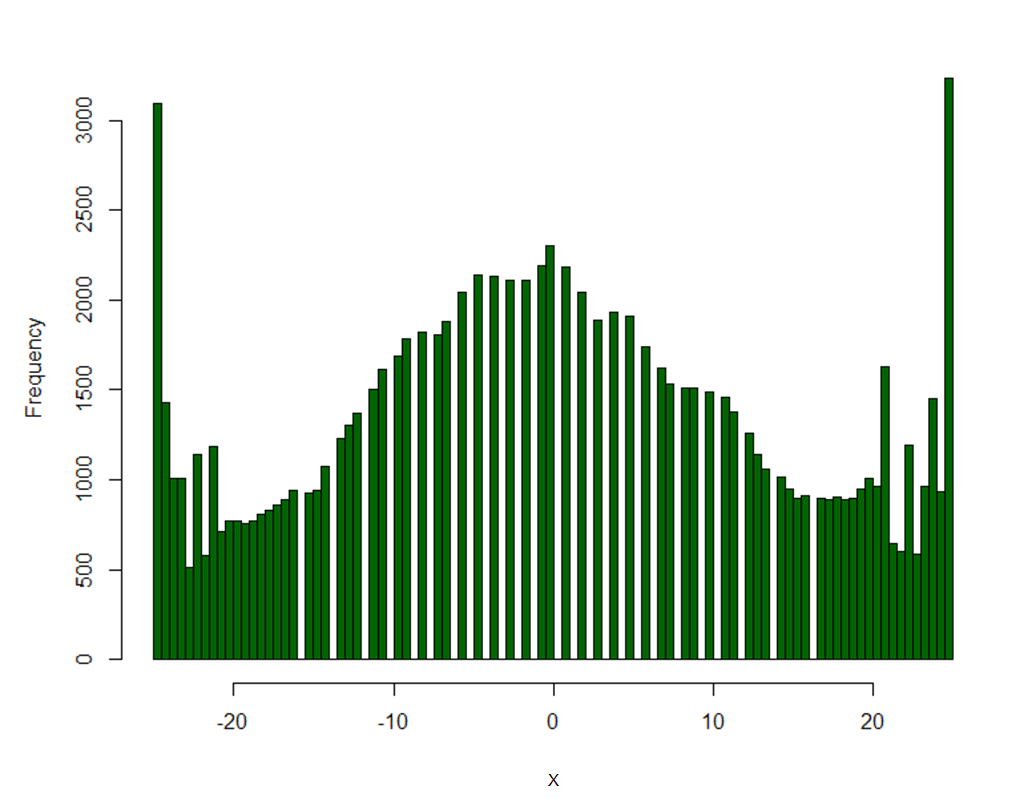}
 \includegraphics[width=\linewidth/2 - \linewidth/70]{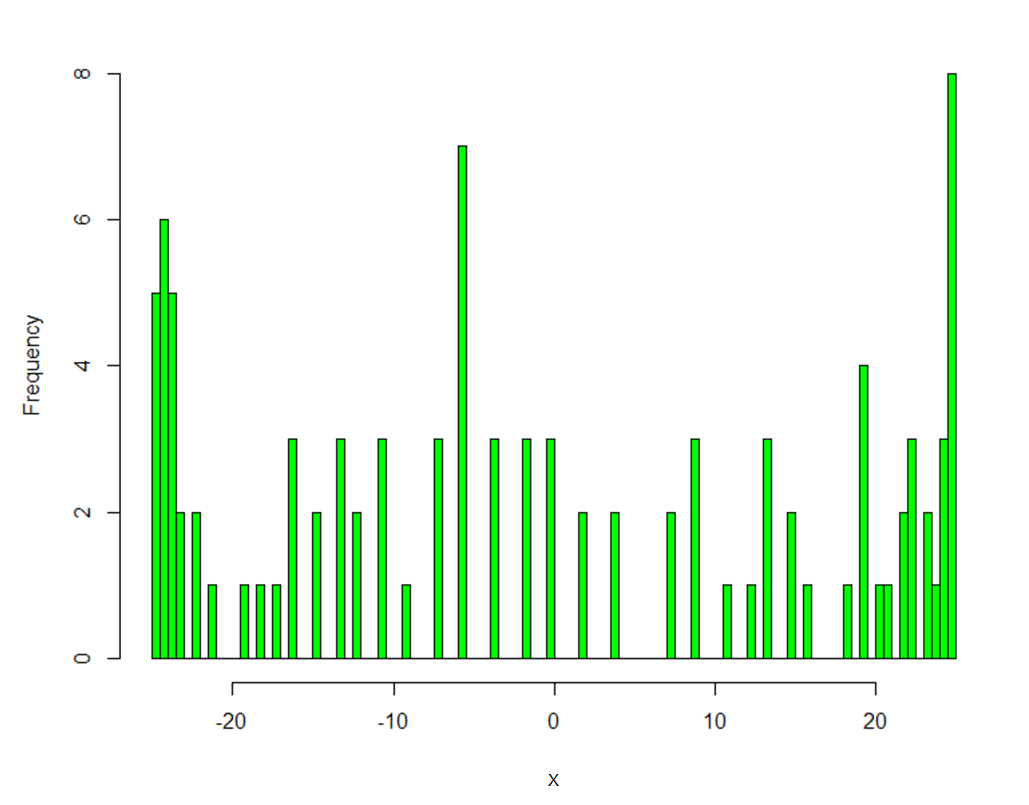}\\
   \textbf{\footnotesize \noindent{
\quad (a). $\boldsymbol{t \in [0, 1000]}$, \quad \quad \quad \quad \quad \quad \quad \quad (b). $\boldsymbol{t =1000}$.
}}
   \caption{Histogram of 10,000 Simulations of 1,000 Step BGC 1-Dimensional It\^{o} Diffusions Showing Discretization or Banding, for Sampling $dW_t$ from a Binomial Distribution}
   \label{Fig:10000Simulationsof1000StepUnconstrained2DimensionalItoDiffusionsWithoutColoration}
\flushleft
   \textbf{\footnotesize \noindent{Note that this is not a density plot as we do not want any Kernal density estimation (KDE) to approximate the distribution.
Increasing the bin size would similarly approximate the histogram.
The main point is that when we consider all points along all paths in (a), we see gaps where the paths do not visit, or visit infrequently.
When we only detect the paths at the end of our timeframe in (b), we see many more gaps and that the gaps are not as homogeneously dispersed as in (a).
}}
\end{figure}
\FloatBarrier

\bigskip \noindent
Figure \ref{Fig:10000Simulationsof1000StepUnconstrained2DimensionalItoDiffusionsWithoutColoration} shows that 1). reflection occurs at the barriers as seen by the peaks on either side of the distribution (more prominent in (a)), and 2). discretization or banding occurs at prominent local times that are contracted together due to the BGC function $\Psi (X_t, t)$.
Figure \ref{Fig:10000Simulationsof1000StepUnconstrained2DimensionalItoDiffusionsWithoutColoration}(a) shows similar `corrugation' as in the continuous case in Figure \ref{Fig:10000Simulationsof1000Step1DimensionalItoDiffusionsWith32ReflectiveBarriersClean}.
Figure \ref{Fig:10000Simulationsof1000StepUnconstrained2DimensionalItoDiffusionsWithoutColoration}(b) shows that at $t=T$, the BGC It\^{o} diffusion is less likely to be at the origin and more likely to be near the barrier(s).
To further demonstrate the characteristics of BGCSP, a detailed plot of 10,000 simulations are shown in Figures \ref{Fig:200ItoDiffusionsFromContinuousDistribution} and \ref{Fig:1000BGCItoDiffusionsSampledFromContinuousDistribution}.

\begin{figure}[htb]
  \centering
  \begin{turn}{90}
  \begin{minipage}{8.5in}
  \centering
   \includegraphics[width=\linewidth]{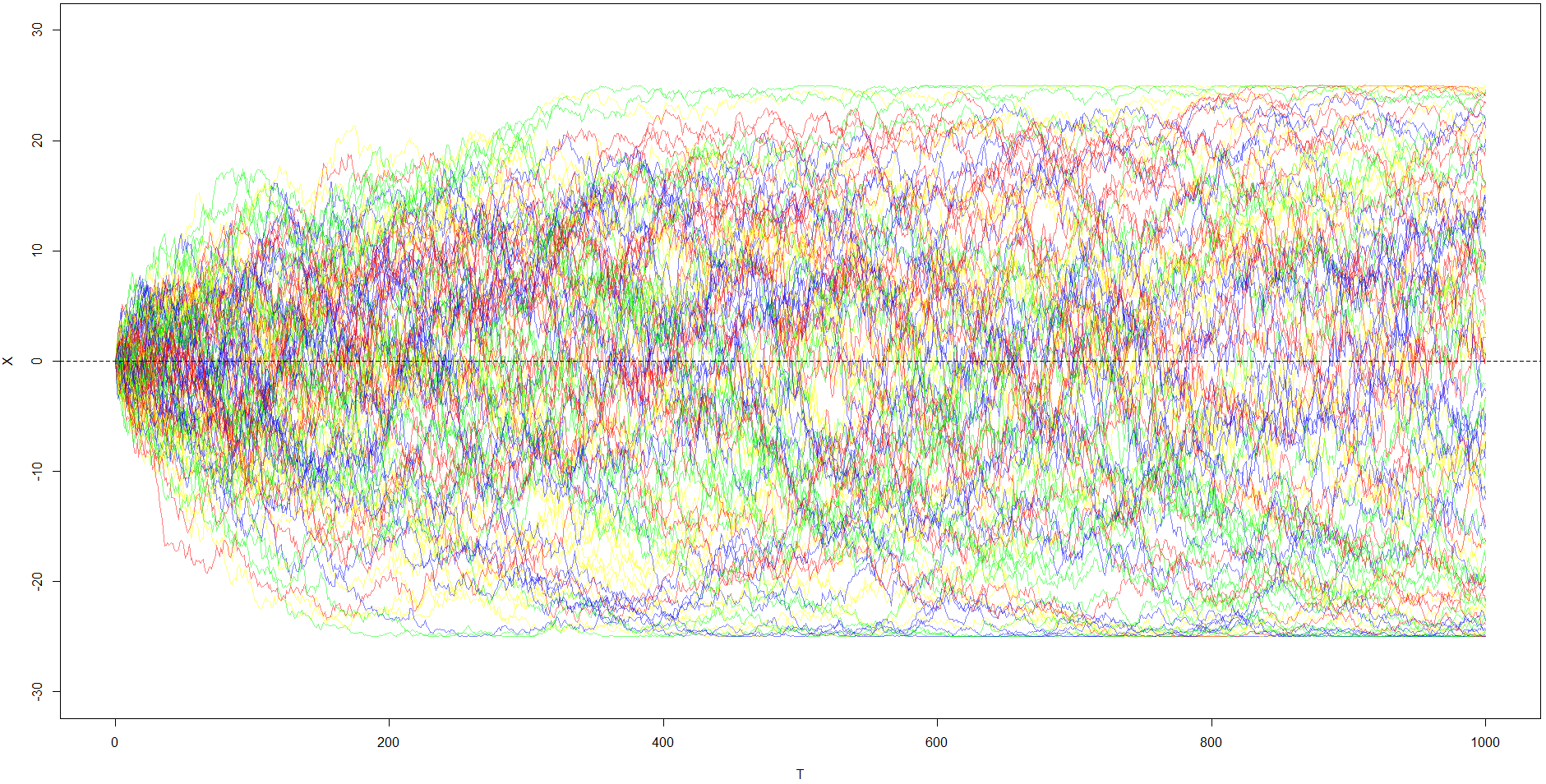}
 \caption{200 BGC It\^{o} Diffusions Sampled from the Standard Normal Distribution}
  \label{Fig:200ItoDiffusionsFromContinuousDistribution}
\flushleft
   \textbf{\footnotesize \noindent{The hidden barriers emerge as the It\^{o} diffusions reach their equilibrium point(s).
The reflectiveness of $\boldsymbol{\Psi(X_t, t)}$ induces the hidden barriers $\mathfrak{B}_L$, $\mathfrak{B}_U$ and they become more visible over time. 
}}
  \end{minipage}
  \end{turn}
\end{figure}
\FloatBarrier

\begin{figure}[htb]
  \centering
  \begin{turn}{90}
  \begin{minipage}{8.5in}
  \centering
   \includegraphics[width=\linewidth]{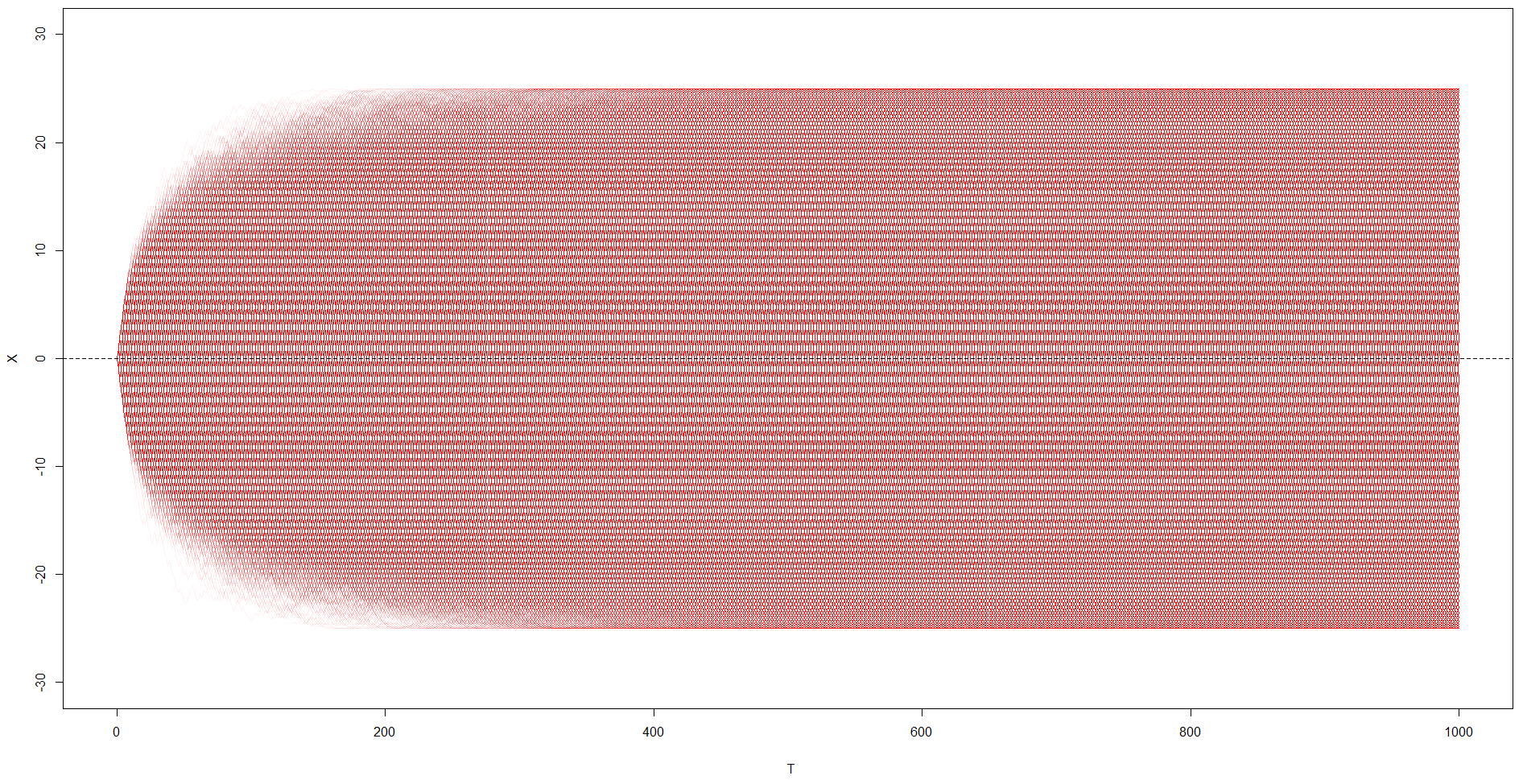}
 \caption{1,000 BGC It\^{o} Diffusions Sampled from the Standard Normal Distribution}
  \label{Fig:1000BGCItoDiffusionsSampledFromContinuousDistribution}
\flushleft
   \textbf{\footnotesize \noindent{Even with continuous $\boldsymbol{dW_t}$, there are regions where the paths do not visit as often.
This would be more noticeable as with our previous research \cite{TarantoKhan2020_1}, \cite{TarantoKhan2020_2}, if only the markers were used rather than the markers and the lines, as is the case in this plot.
We notice a certain amount of banding in regions on either side of the origin.
These regions become closer and closer together the further $\boldsymbol{X}$ is away from the origin.
The use of opacity facilitates these regions to be more visible, otherwise, it would be a uniformly red rectangular region.
}}
  \end{minipage}
  \end{turn}
\end{figure}
\FloatBarrier

\bigskip
\section{Conclusions}

\noindent
This paper has extended the previous theoretical research on BGCSP by comparing them to a type of multi-skew Brownian motion (M-SBM).
This was achieved both theoretically by leveraging existing research, and heuristically by generating new simulations.
Working within the M-SBM framework, we proved one Lemma and two Theorems for BGCSPs.
This research provides a richer framework in which the semipermeable barriers are modulated in a non-constant manner over distance $X$, allowing for a new constraining regime that is more complex than the Ornstein-Uhlenbeck process (OUP) and yet still related to it.
BGCSPs have applications in many fields requiring the constraining of the underlying stochastic process in a gradual manner where the two ultimate reflective barriers are not known in advance.
\FloatBarrier

\bigskip
\bibliographystyle{apalike}

\bibliography{References} 

\begin{thebibliography}{}

\bibitem[Appuhamillage and Sheldon, 2012]{AppuhamillageSheldon2012}
Appuhamillage, T. and Sheldon, D. (2012).
\newblock First passage time of skew brownian motion.
\newblock {\em J. Appl. Probab.}, 49(3):685--696.

\bibitem[Atar and Budhiraja, 2015]{AtarBudhiraja2015}
Atar, R. and Budhiraja, A. (2015).
\newblock On the multi-dimensional skew brownian motion.
\newblock {\em Stochastic Process. Appl.}, 125(5):1911--1925.

\bibitem[Ball and Roma, 1998]{BallRoma1998}
Ball, C.~A. and Roma, A. (1998).
\newblock Detecting mean reversion within reflecting barriers: application to
  the european exchange rate mechanism.
\newblock {\em Applied Mathematical Finance}, 5(1):1--15.

\bibitem[Bramson et~al., 2010]{BramsonDaiHarrison2010}
Bramson, M., Dai, J., and Harrison, J. (2010).
\newblock Positive recurrence of reflecting brownian motion in three
  dimensions.
\newblock {\em The Annals of Applied Probability}, 20(2):753--83.

\bibitem[Budhiraja and Dupuis, 2003]{BudhirajaDupuis2003}
Budhiraja, A. and Dupuis, P. (2003).
\newblock Large deviations for the emprirical measures of reflecting brownian
  motion and related constrained processes in r+.
\newblock {\em Electronic Journal of Prob.}, 8.

\bibitem[Dereudre et~al., 2015]{DereudreMazzonettoRoelly2015}
Dereudre, D., Mazzonetto, S., and Roelly, S. (2015).
\newblock An explicit representation of the transition densities of the skew
  brownian motion with drift and two semipermeable barriers.
\newblock {\em arXiv preprint arXiv:1509.02846}.

\bibitem[Dua et~al., 1976]{DuaKhadilkarSen1976}
Dua, S., Khadilkar, S., and Sen, K. (1976).
\newblock A modified random walk in the presence of partially reflecting
  barriers.
\newblock {\em J. Appl. Prob.}, 13:169--175.

\bibitem[Gairat and Shcherbakov, 2017]{GairatShcherbakov2017}
Gairat, A. and Shcherbakov, V. (2017).
\newblock Density of skew brownian motion and its functionals with application
  in finance.
\newblock {\em Mathematical Finance}, 27(4):1069--1088.

\bibitem[Gupta, 1966]{Gupta1966}
Gupta, H. (1966).
\newblock Random walk in the presence of a multiple function barrier.
\newblock {\em Journ. Math. Sci.}, 1:18--29.

\bibitem[Harrison and Shepp, 1981]{HarrisonShepp1981}
Harrison, J. and Shepp, L. (1981).
\newblock On skew brownian motion.
\newblock {\em Ann. Probab.}, 9(2):309--313.

\bibitem[Harrison, 1986]{Harrison1986}
Harrison, M. (1986).
\newblock {\em Brownian Motion and Stochastic Flow Systems}.
\newblock John Wiley \& Sons.

\bibitem[Hu et~al., 2012]{HuWangYang2012}
Hu, Q., Wang, Y., and Yang, X. (2012).
\newblock The hitting time density for a reflected brownian motion.
\newblock {\em Computational Economics}, 40(1):1--18.

\bibitem[It\^{o} and McKean, 1965]{ItoMcKean1965}
It\^{o}, K. and McKean, H. (1965).
\newblock {\em Diffusion Processes and Their Sample Paths}.
\newblock Springer, New York.

\bibitem[Krykun, 2017]{Krykun2017}
Krykun, I. (2017).
\newblock Convergence of skew brownian motions with local times at several
  points that are contracted into a single one.
\newblock {\em Journal of Mathematical Sciences}, 221(5):671--678.

\bibitem[LeGall, 1984]{LeGall1984}
LeGall, J. (1984).
\newblock {\em One-dimensional stochastic differential equations involving the
  local times of the unknown process}.
\newblock Springer, Berlin, Heidelberg.

\bibitem[Lehner, 1963]{Lehner1963}
Lehner, G. (1963).
\newblock One-dimensional random walk with a partially reflecting barrier.
\newblock {\em Ann. Math. Stat.}, 34:405--412.

\bibitem[Lejay, 2006]{Lejay2006}
Lejay, A. (2006).
\newblock On the constructions of the skew brownian motion.
\newblock {\em Probab. Surv.}, 3:413--466.

\bibitem[L'{e}pingle, 2009]{Lepingle2009}
L'{e}pingle, D. (2009).
\newblock Boundary behavior of a constrained brownian motion between
  reflecting-repellent walls.

\bibitem[Linetsky, 2005]{Linetsky2005}
Linetsky, V. (2005).
\newblock On the transition densities for reflected diffusions.
\newblock {\em Advances in Applied Probability}, 37(2):435--460.

\bibitem[Lions and Sznitman, 1984]{LionsSznitman1984}
Lions, P. and Sznitman, A. (1984).
\newblock Stochastic differential equations with reflecting boundary
  conditions.
\newblock {\em Communications on Pure and App. Maths}, 37(4):511--537.

\bibitem[Mazzonetto, 2016]{Mazzonetto2016}
Mazzonetto, S. (2016).
\newblock {\em On the exact simulation of (skew) Brownian diffusions with
  discontinuous drift (Doctoral dissertation, Universit\"{a}t Potsdam)}.

\bibitem[Mazzonetto, 2019]{Mazzonetto2019}
Mazzonetto, S. (2019).
\newblock Rates of convergence to the local time of oscillating and skew
  brownian motions.
\newblock {\em arXiv preprint arXiv:1912.04858}.

\bibitem[Ouknine et~al., 2015]{OuknineRussoTrutnau2015}
Ouknine, Y., Russo, F., and Trutnau, G. (2015).
\newblock On countably skewed brownian motion with accumulation point.
\newblock {\em Electron. J. Probab.}, 20(82):1--27.

\bibitem[Percus, 1985]{Percus1985}
Percus, O. (1985).
\newblock Phase transition in one-dimensional random walk with partially
  reflecting boundaries.
\newblock {\em Adv. Appl. Prob.}, 17:594--606.

\bibitem[Portenko, 1976]{Portenko1976}
Portenko, N. (1976).
\newblock {\em Generalized diffusion processes}.
\newblock Springer, Berlin.

\bibitem[Ramirez, 2011]{Ramirez2011}
Ramirez, J. (2011).
\newblock Multi-skewed brownian motion and diffusion in layered media.
\newblock {\em Proceedings of the American Mathematical Society},
  139(10):3739--3752.

\bibitem[Sacerdote et~al., 2016]{SacerdoteTamborrinoZucca2016}
Sacerdote, L., Tamborrino, M., and Zucca, C. (2016).
\newblock First passage times of two-dimensional correlated processes:
  Analytical results for the wiener process and a numerical method for
  diffusion processes.
\newblock {\em Journal of Computational and Applied Mathematics}, 296:275--292.

\bibitem[Taranto and Khan, 2020a]{TarantoKhan2020_5}
Taranto, A. and Khan, S. (2020a).
\newblock Bi-directional grid absorption barrier constrained stochastic
  processes with applications in finance and investment.
\newblock {\em Risk Governance \& Control: Financial Markets \& Institutions},
  10(3):20--33.

\bibitem[Taranto and Khan, 2020b]{TarantoKhan2020_6}
Taranto, A. and Khan, S. (2020b).
\newblock Drawdown and drawup of bi-directional grid constrained stochastic
  processes.
\newblock {\em Journal of Mathematics and Statistics}, 16(1):182--197.

\bibitem[Taranto and Khan, 2020c]{TarantoKhan2020_4}
Taranto, A. and Khan, S. (2020c).
\newblock Gambler’s ruin problem and bi-directional grid constrained trading
  and investment strategies.
\newblock {\em Investment Management and Financial Innovations}, 17(3):54--66.

\bibitem[Taranto and Khan, 2021a]{TarantoKhan2020_7}
Taranto, A. and Khan, S. (2021a).
\newblock Application of bi-directional grid constrained stochastic processes
  to algorithmic trading.
\newblock {\em Journal of Mathematics and Statistics}, 17(1):22--29.

\bibitem[Taranto and Khan, 2021b]{TarantoKhan2020_2}
Taranto, A. and Khan, S. (2021b).
\newblock Hidden geometry of bi-directional grid constrained stochastic
  processes.
\newblock {\em arXiv preprint}.

\bibitem[Taranto et~al., 2020]{TarantoKhan2020_1}
Taranto, A., Khan, S., and Addie, R. (2020).
\newblock Iterated logarithm bounds of bi-directional grid constrained
  stochastic processes.
\newblock {\em arXiv Preprint: Modern StochAstic Models \& ProbleMs Of
  Actuarial MaTHematics (MAMMOTH) Conference}, pages 1--21.

\bibitem[Weesakul, 1961]{Weesakul1961}
Weesakul, B. (1961).
\newblock The random walk between a reflecting and an absorbing barrier.
\newblock {\em Ann. Math. Statist.}, 32:765--769.

\end{thebibliography}

\end{document}